\pdfoutput=1
\RequirePackage{ifpdf}
\ifpdf 
\documentclass[pdftex]{sigma}
\else
\documentclass{sigma}
\fi

\usepackage[all]{xy}

\def\Wt(#1){m_{#1}} \let\wh=\widehat

\DeclareMathOperator{\Aut}{Aut}
\DeclareMathOperator{\Hilb}{Hilb}
\DeclareMathOperator{\Ker}{Ker}
\DeclareMathOperator{\Spec}{Spec}
\DeclareMathOperator{\cod}{cod}
\def\cSym{\mathop{\rm Sym}\nolimits}

\let\bb=\mathbb

\let\I=\mathbf
\let\into=\hookrightarrow
\let\mc=\mathcal
\let\onto=\twoheadrightarrow
\let\ox=\otimes

\let\wt=\widetilde

\let\vf=\varphi
\let\x=\times

\def\risom{\buildrel\sim\over{\smashedlongrightarrow}}
 \def\smashedlongrightarrow{\setbox0=\hbox{$\longrightarrow$}\ht0=1.25pt\box0}

\numberwithin{equation}{section}

\newtheorem{thm}{Theorem}[section]
\newtheorem{cor}[thm]{Corollary}
\newtheorem{lem}[thm]{Lemma}
\newtheorem{prp}[thm]{Proposition}

\theoremstyle{definition}
\newtheorem{dfn}[thm]{Definition}

\newtheorem{rmk}[thm]{Remark}
\newtheorem{sbs}[thm]{}

\let\x=\times

\begin{document}
\allowdisplaybreaks

\newcommand{\arXivNumber}{2202.11611}

\renewcommand{\thefootnote}{}

\renewcommand{\PaperNumber}{059}

\FirstPageHeading

\ShortArticleName{Node Polynomials for Curves on Surfaces}

\ArticleName{Node Polynomials for Curves on Surfaces\footnote{This paper is a~contribution to the Special Issue on Enumerative and Gauge-Theoretic Invariants in honor of Lothar G\"ottsche on the occasion of his 60th birthday. The~full collection is available at \href{https://www.emis.de/journals/SIGMA/Gottsche.html}{https://www.emis.de/journals/SIGMA/Gottsche.html}}}

\Author{Steven KLEIMAN~$^{\rm a}$ and Ragni PIENE~$^{\rm b}$}

\AuthorNameForHeading{S.~Kleiman and R.~Piene}

\Address{$^{\rm a)}$~Room 2-172, Department of Mathematics, MIT,\\
\hphantom{$^{\rm a)}$}~77 Massachusetts Avenue, Cambridge, MA 02139, USA}
\EmailD{\href{mailto:kleiman@math.mit.edu}{kleiman@math.mit.edu}}

\Address{$^{\rm b)}$~Department of Mathematics, University of Oslo,\\
\hphantom{$^{\rm b)}$}~PO Box 1053, Blindern, NO-0316 Oslo, Norway}
\EmailD{\href{mailto:ragnip@math.uio.no}{ragnip@math.uio.no}}

\ArticleDates{Received February 24, 2022, in final form July 28, 2022; Published online August 02, 2022}

\Abstract{We complete the proof of a theorem we announced and partly proved in [\textit{Math. Nachr.} \textbf{271} (2004), 69--90, math.AG/0111299]. The theorem concerns a family of curves on a~family of surfaces. It has two parts. The first was proved in that paper. It describes a natural cycle that enumerates the curves in the family with precisely $r$ ordinary nodes. The second part is proved here. It asserts that, for $r\le 8$, the class of this cycle is given by a~computable universal polynomial in the pushdowns to the parameter space of products of the Chern classes of the family.}


\Classification{14N10; 14C20; 14H40; 14K05}

\begin{flushright}
\it Happy 60th, Lothar
\end{flushright}

\renewcommand{\thefootnote}{\arabic{footnote}}
\setcounter{footnote}{0}

\section{Introduction}

This paper is the fourth in a series about enumerating nodal curves on
smooth complex surfaces. Here we complete the proof of Theorem~2.5 on p.~74 in~\cite{K--P}. It has two parts. The first was proved in~\cite{K--P}. It describes a natural cycle $U(r)$ on the parameter space
of a family of pairs of a~surface and of a curve on it; $U(r)$
enumerates the curves with precisely $r$ ordinary nodes. The second part
is proved here. It asserts that, for $r\le 8$, the class $[U(r)]$ is
given by a computable universal polynomial in the pushdowns of products
of the Chern classes of the family.

The second part was not proved in \cite{K--P}, because we believed our
approach, inspired by Vainsencher's paper~\cite{IV}, would eventually
yield an algorithm for computing the entire polynomial for $[U(r)]$ not
only for $r \le 7$, but also for $r = 8$ and perhaps for all~$r$. So we
chose to publish only the construction of $U(r)$ and to postpone the
rest. Unfortunately, we were too optimistic. Thus here we work out an
ad hoc determination of the polynomial for $r=8$; specifically, we show
that the ``correction term'' is independent of the family, and so can be
found by working out a particular example, such as we did in
\cite[Example~3.8, p.~80]{K--P}.

In \cite[Remark~2.7, p.~74]{K--P} we conjectured that the class
$[U(r)]$ is given for all $r$ by a~universal polynomial in certain
classes $y(a,b,c)$, defined here in Section~\ref{sbs:derived}, which
are pushdowns of products of the (relative) Chern classes of the family.
Moreover, this polynomial should be of the form $P_r(a_1,\dots,a_r)/r!$,
where $P_r(a_1,\dots,a_r)$ is the $r$th (complete) Bell polynomial and
the~$a_i$ are linear polynomials in the $y(a,b,c)$.

G\"ottsche \cite{gottsche} had already conjectured the special case
where the pairs consist of a fixed surface and of the divisors in a
linear system. This celebrated conjecture was proved independently by
Tzeng \cite{Tzeng} and Kool--Shende--Thomas~\cite{K-S-T}. For the
history of the case of plane curves, see \cite[Remark~3.7, p.~78]{K--P}
and the more recent~\cite{Lucia}.

A part of our conjecture has now been proved by Laarakker
\cite[Theorem~A, p.~4921]{Laarakker}. He~defined a cycle $\gamma(r)$
which, under suitable genericity assumptions on the family, is supported
on~$U(r)$, and its class $[\gamma(r)]$ is given by a universal
polynomial in the $y(a,b,c)$. Although he did not prove that the
polynomial is Bell (see his footnote \cite[p.~4918]{Laarakker}), he
did prove that $[\gamma(r)]$ is ``multiplicative'' when the family of
surfaces is a direct sum of families over the same base (see
\cite[Lemma~5.5 and Remark~5.6, p.~4936]{Laarakker}). When the family
is trivial, G\"ottsche had observed that this multiplicative property
implies the polynomial is Bell.
However, when the family is
nontrivial, the multiplicative property is insufficient.

In \cite{K--P} we applied our theorem in several enumerations involving
nontrivial families of surfaces, including the family of all planes in
$\mathbb P^4$. In \cite[Theorem~B, p.~4922]{Laarakker} Laarakker proved
that the number of $r$-nodal plane curves of degree $d$ in $\mathbb P^3$
meeting the appropriate number of general lines, is given by a universal
polynomial in $d$ of degree $\le 9+2r$. Moreover, he explicitly computed
the polynomial for $r\le 12$. In~\cite{RitwikPaulSingh} Mukherjee,
Paul, and Singh did the same; they obtained a recursive formula, and
verified that their results agree with Laarakker's. In~\cite{DasRitwik}
Das and Mukherjee treated the case where the curves may have one
additional nonnodal singularity. In~\cite{RitwikSingh} Mukherjee and
Singh did the same for rational curves.

In \cite[Remark~2.7, p.~74]{K--P} we conjectured that universal
polynomials also enumerate curves with any given equisingularity type.
In \cite[Theorem~10.1, p.~713]{Kazaryan} Kazaryan gave a ``topological
justification'' of our conjecture, but gave no algebraic proof. He
worked with a linear system on a fixed surface, and found several
explicit formulas for curves with singularities of codimen\-sion~$\le 7$.
A few of these formulas had been given in \cite[Theorem~1.2, p.~210]{KP99}. In~\cite{BasuRitwik} Basu and Mukherjee gave recursive
formulas for the number of curves in a linear system on a fixed surface
that have $r$ nodes and one additional singularity of codimension $\le
8-r$. In particular, their formula for $8$-nodal curves recovers ours in
this case; see \cite[Theorem~1.1, p.~210]{KP99}.

In \cite{LiTzeng} Li and Tzeng and, independently in~\cite{Rennemo},
Rennemo proved the existence of universal polynomials enumerating
divisors with isolated singularities of given topological or analytical
types in a trivial family of varieties of arbitrary dimension.

In short, we work here over an algebraically closed field of
characteristic $0$ with pairs $(F/Y,D)$, where $Y$ is a Cohen--Macaulay
algebraic scheme, $F/Y$ is a smooth projective family of surfaces, and
$D$ is a relative, or $Y$-flat, effective divisor on $F$. We let $\pi\colon
F \to Y$ denote the structure map.

 In Section~\ref{sc:IndFam}, given a pair $(F/Y,D)$, we recall from
\cite[pp.~226--227]{KP99} the construction and elementary properties
of its induced pairs $(F_i/X_i,D_i)$. Then we prove some further
properties. Intuitively, $(F_i/X_i,D_i)$ represents a family of curves
that sit on blowups of the surfaces of $F/Y$ and that have one less
$i$-fold point.

In Section~\ref{sc:virtual}, the main results are Lemmas~\ref{lemZiso1}
and~\ref{lemZiso2}, which concern properties of certain subschemes
of the relative Hilbert scheme $\Hilb_{D/Y}^{3r}$.
 In Section~\ref{sec:inters}, we develop some results of bivariant
intersection theory for use in the subsequent sections. Our treatment
here generalizes and improves our shorter one in~\cite{KP99}.

In Section~\ref{sc:thm}, we state the main theorem, Theorem~\ref{th:5-1}.
Then we prove a key recursion relation; we prove the theorem for $r\le
7$; and we explain what more is needed for $r=8$. The difficulty is
that the induced pair $(F_2/X_2,D_2)$ does not satisfy the hypotheses of
the theorem, as $D_2/X_2$ has nonreduced fibers in codimension $7$
above the relative quadruple-point locus~$X_4$ of~$D/Y$.

 Therefore, the recursion that works for $r\le 7$ must be corrected
accordingly. In Section~\ref{sc:corr} we find an expression for the
correction term, and in Section~\ref{sc:indep} we prove that the
correction term is equal to $C[X_4]$ for some integer $C$ that is
independent of the given $(F/Y,D)$. Our proof illustrates the advantage
of developing intersection theory over any universally catenary
Noetherian base. Thus, to complete the proof of the theorem, it
suffices to compute the integer~$C$ in a particular case, such as that
of $8$-nodal quintic plane curves, which we did in
\cite[Example~3.8, p.~80]{K--P}.

However, our proof requires an additional genericity hypothesis: the
analytic type of a fiber of~$D$ at an ordinary quadruple point must not
remain constant along any irreducible component of~$X_4$. This
hypothesis comes into play at just one spot in the proof of Lemma~\ref{lm:versal} to ensure a~certain map is flat. We believe that Lemma~\ref{lm:versal} and Theorem~\ref{th:5-1} hold without this hypothesis.
At any rate, the hypothesis is usually fulfilled in practice.

\section{The induced pairs}\label{sc:IndFam}
 The induced pairs $(F_i/X_i,D_i)$ of a given pair $(F/Y,D)$ play
a central role in the present work. So, in this section, we recall the
theory and develop it further. Here $F$ and $Y$ need only be
Noetherian, and $F/Y$ need only be of finite type.

\begin{sbs}{}{\bf The induced pairs.}\label{sbsIndTr}
 From \cite[pp.~226--227]{KP99}, let's recall the construction and
elementary properties of the induced pairs, but make a few minor changes
appropriate for the present work.

Denote by $p_j\colon F\x_Y F\to F$ the $j$th projection, by $\Delta \subset
F\x_Y F$ the diagonal subscheme, and by $\mc I_\Delta$ its ideal. Say
$D$ is defined by the global section $\sigma$ of the invertible sheaf
$\mc O_F(D)$. Then $\sigma$ induces a section $\sigma_i$ of the sheaf
of relative twisted principal parts,
 \begin{equation}\label{eqdfpp}
 \mc P_{F/Y}^{i-1}(D)
 :=p_{1*}\bigl(p_2^*\mc O_F(D)\big/\mc I_\Delta^i\bigr)
\qquad \text{for}\quad i\ge1.
\end{equation}
Take the scheme of zeros of $\sigma_i$ to be $X_i$, and set $X_0:=F$.

Then $X_1=D$. Further, a geometric point of $X_i$, that is, a map
$\xi\colon \Spec(K)\to X_i$, where~$K$ is an algebraically closed field, is
just a geometric point $\xi$ of $F$ at which the fiber $D_{\pi(\xi)}$ has
multiplicity at least $i$. Also, as $i$ varies, the $X_i$ form a
descending chain of closed subschemes.

The sheaf $\mc P_{F/Y}^{i-1}(D)$ fits into the exact sequence,
\begin{equation*}
 0\to\cSym^{i-1}\Omega^1_{F/Y}(D)\to\mc P_{F/Y}^{i-1}(D)
 \to\mc P_{F/Y}^{i-2}(D)\to0,
 \end{equation*}
{\sloppy where the first term is the symmetric power of the sheaf of relative
differentials, twisted by~$\mc O_F(D)$. Hence $\mc
P_{F/Y}^{i-1}(D)$ is locally free of rank $\binom{i+1}2$ by induction on
$i$. Therefore, at each scheme point $x\in X_i$, we have
\begin{equation}\label{eq61a}
\cod_x(X_{i},F)\le\binom{i+1}{2},
\end{equation}}\noindent
where, as usual, $\cod_x(X_{i},F)$ stands for the minimum
$\min(\dim\mc O_{F,\eta})$ as $\eta$ ranges over the generizations of
$x$ in $X_{i}$. If $\cod_x X_{i}=\binom{i+1}2$ and if $Y$ is
Cohen--Macaulay
at $\pi (x)$, then, since $F/Y$ is smooth, $X_{i}$ is a local complete
intersection in $F$ at $x$, and is Cohen--Macaulay
at $x$.

Denote by $\beta\colon F'\to F\x_YF$ the blowup
along $\Delta$, and by $E$
the exceptional divisor. Set $\vf':=p_1\beta$ and $\pi':=p_2\beta$.
Then $\pi'\colon F'\to F$ is again a smooth family of surfaces, and
projective if $\pi$ is; in fact, over a point $\xi$ of $F$, the fiber
$F'_\xi:= \pi'^{-1}(\xi)$ is just the blowup (via $p_1)$ of the fiber
$F_{\pi (\xi)}:= \pi^{-1}\pi (\xi)$ at $\xi$. For each $i$, set
$F_i:={\pi'}^{-1}(X_i)$, and denote by $\pi_i\colon F_i\to X_i$ the
restriction of $\pi'$. In sum, we have this diagram:
\[
\xymatrix{
F\ar[d]_{\pi} & \ar[l]_{\!\!p_1} \ar[d]_{p_2} F\times_Y F & \ar[l]_{\quad \ \beta} \ar[d]_{\pi'} F' & \ar@{_{(}->}[l]_{\ \beta'_i} \ar[d]_{\pi_i} F_i\\
Y & \ar[l]_{\pi} F & \ar@{=}[l] F & \ar@{_{(}->}[l]_{\ \beta_i}X_i.
}
\]

 In addition, given $r\ge1$, set
\begin{equation}\label{eq61b}
 r_i:=r-{\binom{i+1}2}+2,\qquad
 D'_i:={\vf'}^{-1}D-iE\qquad \text{and}\qquad
 D_i:=D'_i\big|_{F_i}.
\end{equation}
As $F'$ has no associated points on $E$, the subscheme ${\vf'}^{-1}D$
is an effective divisor; so $D'_i$ is a~divisor on $F'$. If $i\ge1$,
then
\begin{equation*}
 D'_i=D'_{i-1}-E.
\end{equation*}
 In \cite[p.~227]{KP99}, we proved the second assertion of the next
lemma. Taking a little more care, we now prove the first too. Later,
in Lemma~\ref{lm:4-3}, we relate $X_i$ and~$r_i$.
 \end{sbs}

\begin{lem}\label{lm:4-2}
For each $i\ge1$, the subscheme $X_i$ of $F$ is the largest subscheme
over which $D'_i$ is effective. Furthermore, $D_i:=D'_i\big|_{F_i}$ is
relative effective on $F_i/X_i$.
\end{lem}
\begin{proof}
By definition of $X_i$, a $Y$-map $t\colon T\to F$ factors through $X_i$ iff
$t^*\sigma_i=0$. Now, $\mc P_{F/Y}^{i-1}(D)$ is locally free on $F$, so
flat over $Y$; hence, $(1\x t)^*\mc I_\Delta^i\mc O_{F\x_Y F}(p_1^*D)$ is
a subsheaf of $(1\x t)^*\mc O_{F\x_Y F}(p_1^*D)$ owing to display
(\ref{eqdfpp}). Therefore, $t^*\sigma_i=0$ iff $(1\x t)^*p_1^*
\sigma\colon \mc O_{F\x_Y T}\to(1\x t)^*\mc O_{F\x_Y F}(p_1^*D)$ factors
through that subsheaf.

Let $q\colon F\x_Y T\to F$ denote the projection. Then $q=p_1(1\x t)$. So
\[(1\x t)^*\mc O_{F\x_Y F}(p_1^*D) = \mc O_{F\x_Y T}(q^*D).\]
Let $\Gamma\subset F\x_Y T$ be the graph subscheme of $t$, and $\mc
I_\Gamma$ its ideal. Then $(1\x t)^{-1}\Delta=\Gamma$. Hence $(1\x
t)^*\mc I_\Delta^i = \mc I_\Gamma^i$.
Therefore, $t^*\sigma_i=0$ iff $q^* \sigma\colon \mc O_{F\x_Y T}\to\mc
O_{F\x_Y T}(q^*D)$ factors through $\mc I_\Gamma^i\mc O_{F\x_Y T}
(q^*D)$.

Set $F'_T:=F'\x_FT$ and $\beta_T:=\beta\x_FT$. Then $\beta_T\colon F'_T\to
F\x_YT$ is the blowup of $F\x_Y T$ along $\Gamma$ as $(1\x t)^*\mc
I_\Delta^i = \mc I_\Gamma^i$. Set $E_T:=E\x_FT$. Then $E_T$ is the
exceptional divisor. Trivially, $I_\Gamma^i \mc O_{F'_T}=\mc
O_{F'_T}(-iE_T)$. However, $I_\Gamma^i\risom(\beta_T)_*\mc
O_{F'_T}(-iE_T)$ since $\Gamma$ is a local complete intersection;
see
\cite[display~(6), p.~601]{EGK}; so the projection formula yields
\[
I_\Gamma^i \mc O_{F\x_Y T}(q^*D)
 =\beta_{T*}\mc O_{F'_T}(\beta_T^*q^*D-iE_T).\]
Set $\vf'_T:=q\beta_T$. Then, therefore, $t^*\sigma_i=0$ iff
$\vf_T^{\prime*}\sigma\colon \mc O_{F'_T}\to\mc O_{F'_T}(\vf_T^{\prime*}D)$ factors
through $\mc O_{F'_T}(\vf_T^{\prime*} D-iE_T)$.

Let $\tau\colon F'_T\to F'$ denote the projection. Then $\vf_T^{\prime*}D-iE_T=
\tau^*D'_i$. Therefore, $t^*\sigma_i=0$ iff $\tau^*D'_i$ is
effective. Thus $X_i$ is the largest subscheme of $F$ over which $D'_i$
is effective.

In particular, on every fiber of $\pi_i$, the restriction of $D_i$ is
effective. Furthermore, $\pi_i$ is flat. Hence, $D_i$ is relative
effective. Thus the lemma holds.
\end{proof}

\begin{lem}
 Let $(F/Y,D)$ be a pair. Then forming all of the induced pairs
$(F_i/X_i,D_i)$ commutes with arbitrary base change $g\colon Y'\to Y$.
\end{lem}

\begin{proof} It follows from \cite[Proposition~3.4, p.~422]{KP09} that the
formation of $F^{(1)}$ and $E^{(1)}$ commutes with base change. Set
$g'\colon F\times_YY'\to F$.
 By \cite[Proposition~16.4.5, p.~19]{EGAIV4}, we have ${g'}^*\mc
P^{i-1}_{F/Y}(D)=\mc P^{i-1}_{F\times_YY'/Y'}\big({g'}^{-1}(D)\big)$, and the section
$\sigma_i$ pulls back to the corresponding section~$\sigma'_i$. Hence
the zero scheme of $\sigma'_i$ is equal to $X_i\times_YY'$.
\end{proof}

\begin{dfn}
Let $Y(\infty)$ denote the subset of $Y$ whose geometric points are those
$\eta$ of $Y$ whose fiber $D_\eta$ is not reduced.

Fix a minimal Enriques diagram $\I D$; see \cite[Section~2, p.~213]{KP99}. Denote by $Y(\I D)$ the subset of $Y$ whose geometric points
are those $\eta$ whose fiber $D_\eta$ has diagram~$\I D$.
 \end{dfn}

\begin{sbs}{}{\bf Arbitrarily near points.} 
 Recall the following notions, notation, and results. First, as in
\cite[Definition~3.1, p.~421]{KP09}, for $j\ge 0$, iterate the construction
of $\pi'\colon F'\to F$ from $\pi\colon F\to Y$ to obtain $\pi^{(j)}\colon F^{(j)}\to
F^{(j-1)}$ with $\pi^{(0)} := \pi$, with $\pi^{(1)} := \pi'$, and so
forth. By~\cite[Proposition~3.4, p.~422]{KP09}, the $Y$-schemes $F^{(j)}$
represent the functors of arbitrarily near points of $F/Y$; the latter
are defined in \cite[Definition~3.3, p.~422]{KP09}. As in
\cite[Definition~3.1, p.~421]{KP09}, we denote by
$\varphi^{(j)}\colon F^{(j)}\to F^{(j-1)}$ the map equal to the composition of the blowup
and the first projection, and by $E^{(j)}\subset F^{(j)}$ the exceptional divisor.

Given a minimal Enriques diagram $\I D$ on $j+1$ vertices, fix an
ordering $\theta$ of these vertices. Also, let $\I U$ be the unweighted diagram
underlying $\I D$. By~\cite[Theorem~3.10, p.~425]{KP09}, the functor of
arbitrarily near points with $(\I U,\theta)$ as associated diagram is
representable by a $Y$-smooth subscheme $F(\I U,\theta)$ of $F^{(j)}$.

By \cite[Corollary~4.4, p.~430]{KP09}, the group of automorphisms $\Aut(\I U)$
acts freely on $F(\I U, \theta)$. So its subgroup $\Aut(\I D)$, of
automorphisms of $\I D$, does too.
Set
 \begin{equation*}\label{eqYDa}
 Q(\I D):=F(\I U, \theta) \big/{\Aut}(\I D);
\end{equation*}
 it is independent of the choice of $\theta$ by \cite[Theorem~5.7, p.~438]{KP09}.
Set $d:=\deg\I D$.

Form the structure map and the universal injection
of \cite[Theorem~5.7, p.~438]{KP09}:
 \begin{equation*}\label{eqYDb}
 q\colon\ Q(\I D)\to Y \qquad \text{and} \qquad \Psi\colon\ Q(\I D)\to\Hilb^d_{F/Y};
 \end{equation*}
 in fact, $\Psi$ is a an embedding in characteristic~0. The
construction and study of $\Psi$ is based on the modern theory of
complete ideals. Finally, set
 \begin{equation}\label{eqYDc}
 G(\I D):=\Hilb^d_{D/Y}\x_{\Hilb^d_{F/Y}}Q(\I D).
 \end{equation}
 \end{sbs}

\begin{lem}\label{lemCnstr}
 The sets $Y(\I D)$ and $Y(\infty)$ are constructible; in fact,
$Y(\infty)$ is closed if $F/Y$ is proper. Furthermore, for all $z\in
G(\I D)$ and $y\in Y(\I D)$, we have
 \begin{equation}\label{eq64a}
 \cod_z(G(\I D),Q(\I D))\le d \qquad \text{and} \qquad
	 \cod_y(Y(\I D),Y)\le \cod\I D.
 \end{equation}
 Finally, for only finitely many $\I D$, is either $G(\I
D)\smallsetminus q^{-1}Y(\infty)$ or $Y(\I D)$ nonempty.
\end{lem}
\begin{proof} Note that $Y(\infty)$ is just the image in $Y$ of the set
of $x\in X_2$ at which the fiber of $X_2/Y $ is of dimension at least 1.
This set is closed in $X_2$, so in $F$. Hence $Y(\infty)$ is
constructible; in fact,~$Y(\infty)$ is closed if $\pi$ is proper.

Only finitely many $\I D$ arise from the fibers of $D/Y$; indeed, this
statement is proved in \cite[Lemma~2.4, p.~73]{K--P} without making use
of its blanket hypothesis that $Y$ is Cohen--Macaulay and of finite type
over the complex numbers; that proof just requires $Y$ to be Noetherian.
Thus there are only finitely many $\I D$ such that $Y(\I D)$ is
nonempty; denote the set of these $\I D$ by $\Sigma$.

The subscheme $\Hilb^d_{D/Y}$ of $\Hilb^d_{F/Y}$ is locally cut out by
$d$ equations by \cite[Proposition~4, p.~5]{AIK}. Therefore, the first
bound holds in~(\ref{eq64a}).

The definitions yield $q(G(\I D))\supset Y(\I D)$. Further, take any
$y\in q(G(\I D)) \smallsetminus Y(\infty)$, and let~$\I D'$ be the diagram of
$D_K$, where $K$ is the algebraic closure of $k(y)$. Then the
definitions yield a natural injection $\alpha\colon \I D\into\I D'$ such
that each $V\in\I D$ has weight at most that of $\alpha (V)$.
So~$\deg\I D'>d$ if $y\notin Y(\I D)$. Hence
 \[
 Y(\I D) = q(G(\I D))\smallsetminus \bigl(Y(\infty)\cup\big({\textstyle \bigcup}
 \{q(G(\I D'))\mid \I D'\in\Sigma\text{ and } \deg\I D'>d\}\big)\bigr).
 \]
 But $G(\I D)$ and the $G(\I D')$ are locally closed. Thus $Y(\I D)$ is
constructible.

To prove the second bound in (\ref{eq64a}), note that $G(\I D)$ has a
unique point, $z$ say, lying over the given $y$. Now, $Q(\I D)/Y$ is
smooth of relative dimension $\dim\I D$ by \cite[Theorem~3.10, p.~425]{KP09}.
Thus, as desired,
 \[
 \cod_y(Y(\I D),Y)=\cod_z(G(\I D),Q(\I D))-\dim\I D \le d-\dim \I D = \cod \I D.
 \]

Finally, suppose $G(\I D)\smallsetminus h^{-1}Y(\infty)$ is nonempty.
Then, as we have just seen, there is an injection $\alpha\colon \I D\into\I
D'$, where $\I D'\in\Sigma$, and each $V\in\I D$ has weight at most that
of $\alpha(V)$. But there are only finitely many such $\I D$, as
desired. \end{proof}

\begin{dfn}
We say that $(F/Y,D)$ is \emph{$r$-generic} if
 for every minimal Enriques
diagram $\I D$ and for every $y\in Y(\I D)$, we have
\begin{equation}\label{eq43c}
	\cod_y(Y(\I D),Y) \ge \min(r+1, \cod\I D).
 \end{equation}

 We say that $(F/Y,D)$ is \emph{strongly $8$-generic} if it is
8-generic and if the analytic type of $D_{\pi(x)}$ at an ordinary
quadruple point $x\in X_4$ is not constant along any irreducible
component $Z$ of~$X_4$; that is, the cross ratio of the four tangents at
$x$ is not the same for all $x \in Z$.
 \end{dfn}

\begin{prp}\label{lm:4-3} Fix $r$. Assume that $Y$ is universally
 catenary and that $(F/Y,D)$ is $r$-generic. Then, for each $i\ge
 2$, the induced pair $(F_i/X_i, D_i)$ is $r_i$-generic.
 \end{prp}
\begin{proof}
Fix $i$. Let $\I D'$ be a minimal Enriques diagram. Let $x$ be a
generic point of the closure of $X_i(\I D')$. Then $x\in X_i(\I D')$ as
$X_i(\I D')$ is constructible by Lemma~\ref{lemCnstr} applied with
$(F_i/X_i, D_i)$ and~$\I D'$ for $(F/Y,D)$ and $\I D$. Set $y:=\pi
(x)$. Let $K$ be an algebraically closed field containing~$k(x)$; then
$K$ contains $k(y)$ too.

Consider the curves $D_K$ and $(D_i)_K$. Note $x\in X_i(\I D')$. So
the curve $(D_i)_K$ is reduced, and is obtained from $D_K$ as follows:
blow up $F_K$, at the $K$-point, $x_K$ say, defined by $x$; take the
preimage of $D_K$; and subtract $i$ times the exceptional divisor.
Hence $D_K$ is reduced and of multiplicity either $i$ or $i+1$ at $x_K$.
In the latter case, $(D_i)_K$ contains the exceptional divisor; in~the
former, it doesn't. In either case, let $\I D$ be the diagram of $D_K$.
Then \cite[Proposition~2.8, p.~420]{KP09} yields
 \begin{equation}\label{eq43b}
 \cod(\I D)\ge\cod(\I D')+\binom{i+1}2-2.
 \end{equation}

Since $F/Y$ is flat, the dimension formula yields
 \[
 \dim\mc O_{F,x} = \dim\mc O_{Y,y} + \dim\mc O_{F_y, x}.
 \]
 However, $x$ is the generic point of a component, $X$ say, of the
closure of $X_i(\I D')$; hence, $\dim\mc O_{F,x} = \cod_x(X,F)$. So $y
= \pi (x)$. So $\dim\mc O_{Y,y} = \cod_y(\pi (X),Y)$. Further, $F/Y$ is of
relative dimension 2; so $\dim\mc O_{F_y, x}=2$. Thus
 \begin{equation}\label{eq43f}
	\cod_x(X,F)-2= \cod_y(\pi (X),Y).
 \end{equation}

However, $y:=\pi (x)\in Y(\I D)$. Hence,
\[
\cod_y(\pi (X), Y)\ge \cod_y(Y(\I D),Y).
\]
Combine the last two displays; then (\ref{eq43c}) yields
 \begin{equation}\label{eqxgp}
 \cod_x(X, F) -2 \ge \min(r+1, \cod\I D).
 \end{equation}

Since $Y$ is universally catenary, $F$ is catenary; hence,
\begin{equation}\label{eqcat}
\cod_x(X, X_i) = \cod_x(X, F) - \cod_x(X_i, F).
\end{equation}
Hence (\ref{eqxgp}) and (\ref{eq61a}) yield
 \[
 \cod_x(X, X_i)\ge\min(r+1, \cod\I D)+2-\binom{i+1}2.
 \]
Therefore, (\ref{eq61b}) and (\ref{eq43b}) yield the desired lower bound:
\begin{equation*}
\cod_x(X, X_i)\ge\min(r_i+1, \cod\I D').\tag*{\qed}
\end{equation*}
\renewcommand{\qed}{}
 \end{proof}

\begin{cor}\label{lm66}
Fix $r$. Assume $(F/Y,D)$ is $r$-generic.
Fix $i\ge2$, let $X$ be a component of $X_i$, take $x\in X\smallsetminus Y(\infty)$,
and set $y:=\pi (x)$. Then
\begin{gather}\label{eq43d}
\cod_x(X,F) =\binom{i+1}2, \qquad
 \cod_y(\pi (X),Y) =\binom{i+1}2 -2\qquad \text{if} \quad r_i \ge-1,
 \\
 \label{eq43d'}
 \cod_x(X,F) \ge r+3,\qquad
 \cod_y(\pi (X),Y) \ge r+1 \qquad \text{if} \quad r_i\le-1.
\end{gather}
 \end{cor}

\begin{proof}
 Plainly we may assume $x$ is the generic point of $X$. Let $K$ be an
algebraically closed field containing $k(x)$, so $k(y)$. Then $D_K$ is
reduced as $x\notin Y(\infty)$. Let $\I D$ be the diagram of $D_K$, and
$\I D'$ that of $D_i$. Then $X$ is a component of the closure of
$X_i(\I D')$. So we may appeal to the proof of Proposition~\ref{lm:4-3}.
Note that equation~(\ref{eqcat}) is trivial here, and we do not need $Y$
to be universally catenary.

Since $x\in X_i$, at the corresponding $K$-point, $D_K$ is of multiplicity
at least $i$. Hence $\I D$ has a~root of weight at least $i$. So
$\cod\I D\ge\binom{i+1}2-2$.

Suppose $r_i\ge-1$. Then (\ref{eq61b}) yields $r+1\ge\binom{i+1}2-2$. So
(\ref{eqxgp}) yields
\[
\cod_x(X,F) \ge\binom{i+1}2.
\]
 But the opposite inequality is (\ref{eq61a}), which always holds. So
equality holds. Thus, in (\ref{eq43d}), the first equation holds. The
second follows from it and (\ref{eq43f}).

Suppose $r_i\le-1$ instead. Then $\binom{i+1}2-2\ge r+1$. So $\cod\I
D\ge r+1$. Hence (\ref{eqxgp}) and (\ref{eq43f}) yield (\ref{eq43d'}).
Thus the corollary is proved.
 \end{proof}

\section{Virtual double points}\label{sc:virtual}
 The minimal Enriques diagram $r\I A_1$ consists of $r$ roots of weight
2 and no other vertices. The corresponding scheme $G(r\I A_1)$ is
particularly important, as it is equal to the subscheme of the Hilbert
scheme $\Hilb^{3r}_{D/Y}$ associated to the geometric fibers of $D/Y$
with at least $r$ distinct singular points. Moreover, we need to
consider it for various $(F/Y,D)$ and $r$. So, for clarity, we set
$G(F/Y,D; r):=G(r\I A_1)$.

In this section, we first recall the basic properties of
$G(F/Y,D; r)$, which were treated in \cite[Proposition~5.9,
p.~439]{KP09}. Then we fix $r\ge1$, and assume $(F/Y,D)$ is $r$-generic.
For each $i\ge1$, we find a natural large open subscheme of
\[ H_i:=G(F_i/X_i, D_i; r_i)\]
 such that the associated geometric fibers of $D_i/X_i$ have exactly
$r_i$ nodes. None lies on the exceptional divisor of a fiber of~$F_i$.
Further, adding the exceptional divisor to the fiber of~$D_i/X_i$ yields
a fiber of~$D_{i-1}/X_{i-1}$, and thus establishes an isomorphism from
the preceding open subscheme to a natural open subscheme of~$H_{i-1}$,
which is dense in the preimage of~$X_i$. These results are treated in
Lemmas~\ref{lemZiso1} and~\ref{lemZiso2} below for later use.

\begin{sbs}{}{\bf Subschemes of the Hilbert scheme.} \label{sbsHZr}
Fix $r$. If
$r\ge1$, let $H(r)$ denote the open subscheme of $\Hilb^r_{F/Y}$ over
which the universal family is smooth; in other words, $H(r)$
parameterizes the unions of $r$ distinct reduced points in the geometric
fibers of $F/Y$. By~convention, if $r=0$, then~$H(r)$ and
$\Hilb^r_{F/Y}$ are both equal to $Y$; if $r\le-1$, then both are empty.

If $r\ge 1$, then Proposition~5.9 on p.~439 in \cite{KP09} asserts that
$H(r)=Q(r\I A_1)$ and that the map $\Psi\colon H(r)\to\Hilb^{3r}_{F/Y}$ is
given on $T$-points, where $T$ is a $Y$-scheme, by sending a subscheme
$W$ of $F_T$, say with ideal $\mc I$, to the subscheme $W'$ with ideal
$\mc I^2$ (note that $W'$ is flat, because the standard sequence
{\samepage \[
0\to\mc I/\mc I^2\to\mc O_{W'}\to\mc O_W\to0
\]
is exact and because $\mc I/\mc I^2$ and $\mc O_W$ are flat);
furthermore, $\Psi$ is always an embedding.

}

Consequently, we may view $G(r\I A_1)$ as a subscheme of
$\Hilb^{3r}_{D/Y}$. Set\vspace{-.5ex}
\[
G(F/Y,D; r):=G(r\I A_1)\subset\Hilb^{3r}_{D/Y}
\]
to avoid confusion. Furthermore, set $G(F/Y,D; 0):=Y$, and for
$r\le-1$, set $G(F/Y,D; r)\allowbreak:=\varnothing$. Finally, for an arbitrary
fixed $r$ and for $i\ge0$, set\vspace{-.5ex}
\[
H_i:=G(F_i/X_i, D_i; r_i).
\]
\end{sbs}

\begin{lem}\label{lemZiso1}
 Fix $r\ge 1$. Assume that $(F/Y,D)$ is $r$-generic and that $Y(\infty)$
is empty. Then there is an open subscheme $U\subset Y$ such that $(1)$
for every $y\in Y\smallsetminus U$,
 \begin{equation}\label{eqZisob}\vspace{-.5ex}
 \cod_y(Y\smallsetminus U, Y)\ge r+1
 \end{equation}
 and $(2)$ for every $i\ge1$ with $r_i\ge0$, if we set\vspace{-.5ex}
 \[
 U_i:=\big(\pi^{-1}U\big)\cap (X_i\smallsetminus X_{i+1})\qquad \text{and} \qquad
 V_i:=\big(\pi^{-1}U\big)\cap (X_{i-1}\smallsetminus X_{i+1}),
 \]
 then $U_i$ is dense in $X_i$, and there is a natural isomorphism of
 $F$-schemes
 \begin{equation*}
 \gamma_i\colon \ H_i\x_FU_i\risom H_{i-1}\x_FU_i.
 \end{equation*}
 \end{lem}

\begin{proof}
 Let $U$ be the complement in $Y$ of the union of the closures of those
$Y(\I D)$ with $\cod\I D\ge r+1$. Then $U$ is open since there are only
finitely many nonempty $Y(\I D)$ by Lemma~\ref{lemCnstr}. By~the same
token, (\ref{eq43c}) implies (\ref{eqZisob}). Thus (1) holds.

Fix $i\ge1$ such that $r_i\ge0$. Then $r+1>\binom{i+1}2-2$. Let $X$ be
a component of $X_i$; let $x\in X$, set $y:=\pi (x)$. Then (\ref{eq43d})
yields that $r+1> \cod_y(\pi (X),Y)$; moreover, if $x\in X_{i+1}$, then $
\cod_y(\pi (X_{i+1}),Y) > \cod_y(\pi (X),Y)$. So (\ref{eqZisob}) implies
$\pi (X)\smallsetminus \pi (X_{i+1})\not\subset Y\smallsetminus U$. Hence
$\pi (X_i\smallsetminus X_{i+1})$ meets $U$. Thus $U_i$ is dense in
$X_i\smallsetminus X_{i+1}$.

Let $z\in H_i\x_FU_i$; let $x$ be its image in $U_i$, and
set $y:=\pi (x)$. Let $K$ be an algebraically closed field containing
$k(z)$, so $k(x)$ and $k(y)$ too. Then $D_K$ is reduced since
$Y(\infty)$ is empty, and $D_K$ has multiplicity exactly $i$ at the
$K$-point $x_K$ defined by $x$ since $x\in U_i$, so $x\notin X_{i+1}$.
Hence $(D_i)_K$ is reduced, and does not contain the exceptional divisor
$E_K$.

Let $\I D$ be the diagram of $D_K$, and $\I D'$ that of $(D_i)_K$. By~\cite[Proposition~2.8, p.~420]{KP09}, we have
 \begin{equation}\label{eq67c}
 \cod(\I D)\ge\cod(\I D')+\binom{i+1}2-2,
 \end{equation}
 with equality if and only if $D_K$ has an ordinary $i$-fold point at
$x_K$. Now, $(D_i)_K$ has at least $r_i$ singular points since $z\in
H_i$; hence, formula (2.6.2) in \cite[p.~419]{KP09} yields $\cod(\I
D')\ge r_i$ since $\I D'$ has at least $r_i$ roots, each root has
multiplicity at least $2$, and the summands in that formula
corresponding to the other vertices of $\I D'$ are nonnegative. So the
right-hand side of~(\ref{eq67c}) is at least $r$. However, $r\ge
\cod(\I D)$ since $y\in U$. So equality obtains everywhere. Hence
$D_K$ has an ordinary $i$-fold point at $x_K$. Furthermore, $(D_i)_K$
has exactly $r_i$ singular points, each is an ordinary double point, and
none lies on~$E_K$; also, $(D_i)_K$ and $E_K$ meet transversally in $i$
points.

We define $\gamma_i$ as follows. A $T$-point of its source $H_i\x_FU_i$
is given by a map $T\to U_i$ and a $T$-smooth subscheme $W\subset F'_T$
of relative length $r_i$ whose squared ideal defines a subscheme
$W'\subset F'_T$ contained in $(D_i)_T$. Owing to the discussion above,
in every geometric fiber of $F'_T/T$, the fibers of $(D_i)_T$ and $E_T$
meet transversally in $i$ points. Hence, since $(D_i)_T$ and $E_T$ are
relative effective divisors, their intersection is a $T$-smooth
subscheme $Z\subset F'_T$ of relative length $i$.

Let $Z'$ be the subscheme of $F'_T$ defined by the squared ideal of $Z$.
Then $Z'$ is contained in the sum $(D_i)_T+E_T$, which is equal to
$(D_{i-1})_T$. So $W\cup Z$ is a $T$-smooth subscheme of $F'_T$ of
relative length $r_i+i$, or $r_{i-1}$. And its squared ideal defines a
subscheme of $F'_T$, namely $W'\cup Z'$, which is contained in
$(D_{i-1})_T$. So $W\cup Z$ determines a $T$-point of $H_{i-1}\x_FU_i$,
and the latter scheme is to be the target of $\gamma_i$. We define
$\gamma_i$ by sending $W$ to $W\cup Z$. Plainly, $\gamma_i$ is
injective on $T$-points since $W$ is determined by $W\cup Z$ as the part
off $E_T$.

To prove $\gamma_i$ is surjective on $T$-points, fix a $T$-point of
$H_{i-1}\x_FU_i$. It is given by a map $T\to U_i$ and a $T$-smooth
subscheme $S\subset F'_T$ of relative length $r_{i-1}$ such that its
squared ideal defines a~subscheme $S'\subset F'_T$ contained in
$(D_{i-1})_T$. Then $(D_{i-1})_T=(D_i)_T+E_T$ by (\ref{eq61b}), and
$(D_i)_T$ is relative effective by Lemma~\ref{lm:4-2} since $U_i\subset
X_i$. Let $W$ be the part of $S$ off $E_T$. Plainly, $W$ is a~$T$-smooth subscheme of $F'_T$, and its squared ideal defines a
subscheme contained in $(D_i)_T$.

Consider a geometric point of $T$, say with (algebraically closed) field
$K$. Then $D_K$ is reduced since $Y(\infty)$ is empty, and $D_K$ has
multiplicity exactly $i$ at the center of $K$ since $T$ maps into~$U_i$.
Hence $(D_i)_K$ is reduced, and $(D_i)_K\cap E_K$ is a scheme of length~$i$. Now, $S_K$ is $K$-smooth of length~$r_{i-1}$. Hence $S_K$
consists of $r_{i-1}$ distinct reduced points, of which at most $i$ lie
on $E_K$. So~$W_K$ consists of at least $r_{i-1}-i$, or $r_i$, distinct
reduced points. By~choosing any $r_i$ of them, we obtain a $K$-point of
$H_i\x_FU_i$. But then, by the discussion of such points right after
\eqref{eq67c}, there was no choice: $(D_i)_K$ has exactly $r_i$ singular
points, and all are ordinary nodes. Hence~$W_K$ consists exactly of
$r_i$ distinct reduced points. Thus $W$ is of relative length $r_i$.

Therefore, $W$ defines a $T$-point of $H_i\x_FU_i$.
According to the discussion above, this $T$-point is carried by
$\gamma_i$ to the $T$-point of $H_{i-1}\x_FU_i$ that is given by
$R$, where $R:=W\cup Z$ and $Z:=(D_i)_T\cap E_T$. To prove that $\gamma_i$
is surjective on $T$-points, so bijective on $T$-points, so an
isomorphism, it remains to prove that $R=S$.

The equation $R=S$ may be checked locally over $T$ and locally on $F$.
So we may replace $T$ and $F$ by affine open subsets $\Spec(A)$ and
$\Spec(B)$. Then $B$ is \'etale over a polynomial subring $A[x,y]$.
Let $I\subset B$ denote the ideal of $S$. Shrinking $F$ further if
necessary, we can find an $f\in B$ that generates the ideal of
$(D_{i-1})_T$. Then $f\in I^2$ as $R$ determines a $T$-point of
$H_{i-1}$. Hence~$f$, $\partial f/\partial x$, $\partial f/\partial y\in
I$. But those three elements generate the ideal of $Z:=(D_i)_T\cap E_T$
on a~neighborhood $N$ of $E_T$. Hence $Z\supset S\cap N$. But both $Z$
and $S\cap N$ are $T$-flat of relative length~$i$. Hence $Z = S\cap N$.
But $R$ and $S$ are equal off $E_T$. Thus $R=S$, as desired.
 \end{proof}

\begin{lem}\label{lemZiso2}
 Under the conditions of Lemma~{\rm \ref{lemZiso1}}, the closed
subscheme $H_{i-1}\x_FU_i$ of $H_{i-1}\x_FV_i$ is also open.
 \end{lem}

\begin{proof}
 Consider any $T$-point of $H_{i-1}\x_FV_i$. Let $T'$ be the preimage
of $H_{i-1}\x_FU_i$. It suffices to prove $T'$ is an open subscheme, as
we may take $T=H_{i-1}\x_FV_i$.

Let $\mc I\subset \mc O_T$ denote the ideal of $T'$. Then it suffices
to prove that the stalk $\mc I_t$ vanishes for all $t\in T'$ for the
following reason. Since $\mc I$ is coherent, the $t\in T$, where $I_t$
vanishes form an open subset $T''$. By~hypothesis, $T'\subset T''$.
But, if $t\notin T'$, then $I_t=\mc O_{T,t}$; whence, $T'\supset T''$.
So $T'= T''$. Give $T''$ the induced structure as an open subscheme of
$T$. Then $T'$ is the closed subscheme of $T''$ with ideal $\mc
I\mid_{T''}$. But $\mc I\mid_{T''}=0$. Thus $T'$ is equal to the open
subscheme $T''$.

Given $t\in T'$, to check if $\mc I_t$ vanishes, we may replace $T$ by
$\Spec(\mc O_{T,t})$. Thus we may assume that $T$ is of the form
$\Spec(A)$, where $A$ is local and that $T'$ is nonempty. Then it
suffices to prove the ideal $I\subset A$ of $T'$ vanishes, or
equivalently, $T=T'$.

There exists a flat local homomorphism $A\to B$ such that $B$ is
complete and its residue class field is algebraically closed. Then
$A\to B$ is faithfully flat. So $I$ vanishes if $I\ox_AB$ does. Thus
we may replace $A$ by $B$, and so assume that $A$ is complete and its
residue class field is algebraically closed.

Consider the composition $T\to H_{i-1}\x_FV_i\to V_i$. Via it, $T'$ is
the preimage of $U_i$. Hence $T=T'$ if and only if $(D_i')_T$ is
effective, owing to Lemma~\ref{lm:4-2}. By~the same token,
$(D_{i-1}')_T$ is effective.

Consider the local ring $C$ of $F_T$ at the closed point of the center
of the blowing up $F'_T\to F_T$. Let $\wh C$ be its completion. Since
$C\to\wh C$ is faithfully flat, $(D_i')_T$ is effective if and only if
$(D_i')_T\ox_C\wh C$ is effective.

As $F/Y$ is a smooth family of surfaces and as $\wh C$ is complete with
algebraically closed residue class field, $\wh C$ is a power series
ring; say $\wh C=A[[u,v]]$. Say that the section $T\to F_T$ is defined
by mapping $u$,~$v$ to $a,b\in A$. Replacing $u$,~$v$ by $u-a$, $v-b$, we
may assume that $T\to F$ is defined by mapping $u$,~$v$ to $0$,~$0$.

Let $f$ in $A[[u,v]]$ define the pullback of $D$. Write
$f=f_1+f_2+\dotsb$, where $f_j$ is homogeneous of degree $j$ in $u$ and
$v$. Then $f_j=0$ for $1\le j\le i-2$ since $(D_{i-1}')_T\ox_C\wh C$
is effective. It remains to prove that $f_{i-1}=0$.

To prove that $f_{i-1}=0$, denote the maximal ideal of $A$ by $\I m$,
and write
\begin{equation}\label{equ69h}
f_{i-1}(u,v) = a_1u^{i-1}+a_2u^{i-2}v+\dotsb+a_iv^{i-1}
\qquad \text{with} \quad a_j\in A.
\end{equation}
Then it suffices to prove that
$a_j\in\I m^n$ for all $j$ and $n\ge0$.

Since $T$ maps into $H_{i-1}$, there is a $T$-smooth subscheme $S\subset
F'_T$ of relative length $r_{i-1}$ whose squared ideal defines a
subscheme $S'\subset F'_T$ contained in $(D_{i-1})_T$. Since $A$ has an
algebraically closed residue class field $K$, the fiber $S_K$ consists
of $r_{i-1}$ distinct points. Of them, exactly $i$ lie on $E_K$ according
to our discussion above. Further, the fiber $S_K$ is the singular locus
of $(D_{i-1})_K$, which consists of $r_{i-1}$ ordinary double points,
and $i$ of them constitute $(D_i)_K\cap E_K$, which is a~transverse
intersection.

By replacing $u$ and $v$ with suitable linear combinations of
themselves, we may assume that the $u$-axis is not tangent to $D_K$ at
the center of the blowing up. Now, $A$ is complete; so by Hensel's
lemma, $S$ decomposes into the disjoint sum of $r_{i-1}$ sections. Of
them, $i$ sections meet~$E_T$. Hence, they correspond to $A$-algebra
maps
 \[
 s_j\colon\ A[[u,v]][w]\big/(u-vw)\to A\qquad \text{for} \quad 1\le j\le i.
 \]

Set $b_j:=s_j(v)$ and $c_j:=s_j(w)$. Then, for all $j$, let's check that
\[
b_j\in\I m \qquad \text{and} \qquad f_i(c_j,1)\in \I m.
\]
The first relation holds as the closed points of the sections lie on
$E_K$. The second holds because these same points lie on $(D_i)_K\cap
E_K$. Further, the points are distinct; so mod $\I m$, the $c_j$ are
distinct elements of $K$.

Proceeding by induction on $n\ge1$, suppose that $b_j\in\I m^n$ for all $j$. Set
\[
\bar f(v, w):= f(vw, v)/v^{i-1}.
\]
Then $\bar f$ defines the pullback of $(D_{i-1})_T$. Now, $S\subset
(D_{i-1})_T$. Hence, for each $j$,
\begin{equation*}\label{equ69i}
0 = \bar f(b_j, c_j) = f_{i-1}(c_j, 1) + b_j f_i(c_j, 1) + b_j^2 d_j
\end{equation*}
for some $d_j\in A$. But $f_i(c_j, 1)\in \I m$. Thus $f_{i-1}(c_j, 1)
\in \I m^{n+1}$.

From (\ref{equ69h}), we obtain the following linear system of equations
for the $a_j$:
\[
f_{i-1}(c_j, 1)=a_1c_j^{i-1}+a_2c_j^{i-2}+\dotsb+a_i
	\qquad \text{for} \quad 1\le j\le i.
\]
The coefficient matrix is Vandermonde. Its determinant is invertible
in $A$, as the $c_j$ are distinct mod $\I m$. As $f_{i-1}(c_j, 1) \in
\I m^{n+1}$ for all $j$, solving yields $a_j\in\I m^{n+1}$.

To complete the proof, we must show $b_j \in\I m^{n+1}$ for each
$j$. Set $I_j:=\Ker(s_j)$. Then $\bar f\in I_j^2$ as
$S'\subset(D_{i-1})_T$. Hence $\partial\bar f/\partial w\in
I_j$. Therefore,
\begin{equation}\label{equ69j}
0 = (\partial\bar f/\partial w)(b_j, c_j)
 = (\partial f_{i-1}/\partial u)(c_j, 1)
 + b_j (\partial f_i/\partial u)(c_j, 1) + b_j^2 d_j'
\end{equation}
for some $d_j'\in A$. Now, $a_j\in\I m^{n+1}$; so (\ref{equ69h}) yields
 \[(\partial f_{i-1}/\partial u)(c_j, 1)
 = (i-1)a_1c_j^{i-2}+(i-2)a_2c_j^{i-3}+\dotsb+a_{i-1}\in\I m^{n+1}.\]
But $b_j\in\I m^n$. So (\ref{equ69j}) yields
	$b_j (\partial f_i/\partial u)(c_j, 1) \in\I m^{n+1}$. But
$(c_j, 1)$ is, mod $\I m$, a simple root of~$f_i$; so $(\partial
f_i/\partial u)(c_j, 1) \notin\I m$. Thus $b_j \in\I m^{n+1}$, as
desired.
\end{proof}

\section{Intersection theory}\label{sec:inters}
 For use in the remaining sections, we extend the intersection theory
of bivariant classes developed in \cite[Chapter~17]{Fu} and generalized
over any universally catenary base in \cite[Sections~2 and~3]{Kl85}, in
\cite{Th87}, and in \cite[Chapter~42]{stacks-project}. However, only~\cite{Fu} is cited below.

\begin{sbs}{}{\bf Push down.} 
 Assume that $f\colon X\to Y$ is a map of schemes such that its
orientation class~$[f]$ is defined \cite[Section~17.4, p.~326]{Fu}. If $f$ is
also proper, define an additive map
\begin{equation*}\label{eqlp}
 f_\#\colon\ A^*(X) \to A^*(Y) \qquad \text{by} \quad
 f_\#a := f_*(a\cdot[f]),
 \end{equation*}
 where $f_*\colon A^*(f)\to A^*(Y)$ is the proper push-forward operation
discussed in~$({\bf P}_2)$ on~p.~322 of~\cite{Fu}.
 \end{sbs}

\begin{prp}\label{prCom}
 Let $a,b\in A^*(X)$. Assume that $a$ is a polynomial in Chern classes
of vector bundles on $X$. Then $f_\#a\cdot f_\#b = f_\#b\cdot
f_\#a$.
 \end{prp}

\begin{proof}
 By $({\bf A}_{12})$ on~p.~323 of~\cite{Fu}, product and push-forward
commute; so
 \begin{equation*}\label{eqprCom1}
 f_\#a\cdot f_\#b = f_*(a\cdot[f]\cdot f_*(b\cdot[f])).
 \end{equation*}

Let $p_i\colon X\x_YX\to X$ denote the $i$th projection. Form
the diagram
\begin{equation*}\label{eqprCom2}
\begin{CD}
 X\x_YX @>1>> X\x_YX @>p_2>> X\\
 @VV p_1 V @VV p_1 V @VV f V\\
 X @>1>> X @>f>> Y @>1>> Y.
 \end{CD}
 \end{equation*}
 Apply the projection formula of $({\bf A}_{123})$ of
\cite[p.~323]{Fu} with $f:=f$, with $g:=f$, with $h:=1_Y$, with
$c:=[f]$ and with $d:=b\cdot[f]$. The result is
 \begin{equation*}\label{eqpr2}
 [f]\cdot f_*(b\cdot[f])
 = p_{1*}\bigl(f^*([f])\cdot b\cdot[f]\bigr).
 \end{equation*}

The definitions of $f^*$, of $[f]$, and of $p_2$ yield
$f^*([f])=[p_2]$. But Axiom $({\bf C}_2)$ on p.~320 of~\cite{Fu}
 yields
$[p_2]\cdot b = p_2^*(b)\cdot [p_2]$. Thus
$[f]\cdot f_*(b\cdot[f]) = p_{1*}(p_2^*(b)\cdot [p_2]\cdot[f])$.

Apply this projection formula again, but now with $f:= 1_X$, with $g:=
p_1$, with $h:=f$, with $c:= a$ and with $d:=p_2^*(b)\cdot
[p_2]\cdot[f]$. The result is
 \begin{equation*}\label{eqpr3}
 a\cdot p_{1*}\bigl(p^*_2(b)\cdot [p_2]\cdot[f]\bigr)
 = p_{1*}\bigl(p_1^*(a)\cdot (p_2^*(b)\cdot [p_2]\cdot[f])\bigr).
 \end{equation*}

The formula just before Proposition~17.4.1 on p.~327 of~\cite{Fu} yields
$[p_2]\cdot[f] = [f p_2]$. So the functoriality of pushforwards,
stated in $({\bf A}_2)$ on p.~323 of~\cite{Fu}, yields
 \begin{equation*}\label{eqpr4}
 f_*\bigl(p_{1*}\bigl(p_1^*(a)\cdot p_2^*(b)\cdot
 [p_2]\cdot[f]\bigr)\bigr)
 = (f p_1)_*\bigl(p_1^*(a)\cdot p_2^*(b)\cdot[f p_2])\bigr).
 \end{equation*}
 Putting it all together yields
 \begin{equation}\label{eqpr5}
 f_\#a\cdot f_\#b  = (f p_1)_*\bigl(p_1^*(a)\cdot p_2^*(b)\cdot[f p_2]\bigr).
 \end{equation}

By hypothesis, $a$ is a polynomial in Chern classes of vector bundles on
$F$. But product and pullback commute by property $({\bf A}_{13})$ on
p.~323 of~\cite{Fu}. So $p_1^*(a)$ is the same polynomial in the same Chern classes
of the pullbacks under $p_1$ of those vector bundles. But, as stated
just before Proposition~17.3.2 on p.~325 of~\cite{Fu}, Chern classes commute with all
bivariant classes. Thus $p_1^*(a)\cdot p_2^*(b) = p_2^*(b)\cdot p_1^*(a)$.

Of course, $f p_1=f p_2$. Thus
 \begin{equation*}\label{eqpr6}
 f_\#a\cdot f_\#b
 = (f p_2)_*\bigl(p_2^*(b)\cdot p_1^*(a)\cdot[f p_1]\bigr).
 \end{equation*}
 But $a$ and $b$ are arbitrary in \eqref{eqpr5}; moreover, $p_1$ and $p_2$
may be interchanged. So
 \begin{equation*}\label{eqpr7}
 (f p_2)_*\bigl(p_2^*(b)\cdot p_1^*(a)\cdot[f p_1]\bigr)
 = f_\#b\cdot f_\#a.
 \end{equation*}
 Thus $f_\#a\cdot f_\#b = f_\#b\cdot f_\#a$, as asserted.
 \end{proof}

Assume $g\colon Y'\to Y$, and consider $f'\colon X':=X\times_Y Y' \to Y'$. Then,
\begin{equation}\label{A23}
g^*f_\#=f'_\# g^* \colon\ A^*(X)\to A^*(Y').
\end{equation}
This property results from property $({\bf A}_{23})$ on p.~323 of
\cite{Fu} as follows: \begin{align*}\label{alA23}
 g^*f_\#(a) &= g^*f_*(a\cdot [f]) = f'_*g^*(a\cdot [f])
	 = f'_*(g^*(a)\cdot g^*[f])\\
	&= f'_*(g^*(a)\cdot [f']) = f'_\#(g^*(a)).
 \end{align*}

 Assume that $f\colon X\to Y$ and $g\colon Y\to Z$ are proper and that $[f]$,
$[g]$, and $[gf]$ exist. Then
 \[
 (gf)_\#=g_\# f_\#,
 \]
 since, by \cite[Section~17.4, p.~327]{Fu},
 \[
 (gf)_\#(a)=g_*f_*(a\cdot [gf])=g_*f_*(a\cdot [f] \cdot
	[g])=g_*(f_*(a\cdot [f])\cdot [g])=g_\#(f_\#(a)).
 \]

 \begin{sbs}{}{\bf Blowups.} 
 Let $\iota\colon W\to V$ be a closed, regular embedding of
codimension $d$. Then the orientation class $[\iota]$ is defined. Let
$V'$ denote the blowup of $V$ along $W$, with exceptional divisor~$E$. Then $E=\mathbb P\big(\nu^\vee\big)$, where $\nu$ is the normal bundle of
$W$ in $V$. Set $\xi :=c_1(\mc O_{V'}(-E))$. The map $f\colon E\to W$ is
flat, hence has an orientation class $[f]$. Then, by
\cite[Corollary~4.2.2, p.~75]{Fu}, for $k\ge 1$,
 \begin{equation}\label{blowup}
 f_\#\xi^k=-s_{k-d}\big(\nu^\vee\big),
 \end{equation}
 where $s\big(\nu^\vee\big)=c(\nu)^{-1}$ denotes the Segre class of $\nu^\vee$.
 \end{sbs}

 \begin{sbs}{}{\bf Derived classes.} \label{sbs:derived}
 Consider the setup of Section~\ref{sbsIndTr}. We have the Chern
classes $v:=c_1(\mc O_F(D))$, $w_j:=c_j\big(\Omega^1_{F/Y}\big)$ (for $j=1,2$),
$e:=c_1(\mc O_{F'}(E))\in A^*(F')$, where $E$ is the exceptional divisor
on $F^{\prime}$. Let $\beta\colon E\to \Delta\cong F$ also denote the
restriction of $\beta\colon F'\to F\times_YF$. Recall that $\beta_i\colon X_i\to
F$ and $\beta'_i\colon F_i\to F'$ denote the inclusions. Let $v, w_j, e \in
A^*(F')$ also denote their own pullbacks via $\vf'=p_1\beta$. As in
\cite{KP99,K--P}, we have
 \begin{equation*}
 c_1(\mc O_{F_i}(D_i))=(\beta'_{i})^*(v-ie),\qquad c_1\big(\Omega^1_{F'/F}\big)=w_1+e,\qquad
 \text{and}\qquad c_2\big(\Omega^1_{F'/F}\big)=w_2-e^2.
 \end{equation*}
We also have the following relations:
\begin{equation*}
 e^3+w_1e^2+w_2e=0,\qquad \beta_\#e=0,\qquad
 \text{and}\qquad \beta_\#e^2=-s_0\big(\nu^\vee\big).
\end{equation*}
 The first relation results from the relation
$e^2+w_1e+w_2=0$ on $E$, cf.\ \cite[Remark~3.2.4, p.~55]{Fu}; the
second and third, from equation~(\ref{blowup}).

Let $\iota\colon \Delta \to F\times_Y F$ denote the embedding, which is
regular of codimension $2$, since $\pi\colon F\to Y$ is smooth of relative
dimension $2$. As a class on $F\times_Y F$, we can write
$\beta_\#e^2=-\iota_\# s_0\big(\nu^\vee\big)$.

Set $y(a,b,c):=\pi_\#\big(v^aw_1^bw_2^c\big)\in A^*(Y)$. (Note, by Proposition~\ref{prCom}, multiplying these $y(a,b,c)$ is commutative.) The
corresponding classes for the map $\pi_i\colon F_i\to X_i$ (defined in
Section~\ref{sbsIndTr}) are
\[y_i(a,b,c):={\pi_i}_\#(\beta'_i)^*\big((v-ie)^a(w_1+e)^b\big(w_2-e^2\big)^c\big).\] By
(\ref{A23}), $\beta_i^* \pi'_\# =\pi_{i\#} {\beta'_i}^*$; hence,
 \[
 y_i(a,b,c)=\beta_i^*\pi'_\# \big((v-ie)^a(w_1+e)^b\big(w_2-e^2\big)^c\big).
 \]
We have $\pi'=p_2 \beta$, and
 \begin{align*}
 \beta_\# \big((v-ie)^a(w_1+e)^b\big(w_2-e^2\big)^c\big)&= \beta_\#\big(v^a w_1^b
 w_2^c-Q(i;a,b,c)e^2\big) \\
 &= p_{1}^*\big(v^a w_1^b w_2^c\big)+p_1^*Q(i;a,b,c)\cdot
 \iota_\#s_0\big(\nu^\vee\big),
 \end{align*}
 where $Q(i;a,b,c)$ is the (weighted homogeneous, degree $a+b+2c-2$)
polynomial in $v$, $w_1$, $w_2$ returned by the function in Algorithm~2.3 in \cite[p.~72]{K--P}. Hence, since $[\pi']=[p_2
\beta]=[\beta]\cdot [p_2]$ by \cite[Section~17.4, p.~327]{Fu}), we get
 \[
 \pi'_\# \big((v-ie)^a(w_1+e)^b\big(w_2-e^2\big)^c\big)=p_{2\#}p_{1}^*\big(v^a w_1^b
 w_2^c\big)+p_{2\#}\bigl(p_{1}^*Q(i;a,b,c)\cdot
 \iota_\#s_0\big(\nu^\vee\big)\bigr).
 \]
 Since $p_1\iota=p_2\iota$, we have
 \[ p_{1}^*Q(i;a,b,c)\cdot \iota_\#s_0\big(\nu^\vee\big)
 = \iota_\#\big(\iota^*p_{1}^*Q(i;a,b,c)\cdot s_0\big(\nu^\vee\big)\big)
 = p_{2}^*Q(i;a,b,c)\cdot \iota_\# s_0\big(\nu^\vee\big). \]
 Since $s_0(\nu^\vee)=c_0(\nu)=:1_\Delta\in A^*(\Delta)$ is the class
that acts as identy on $A_*(\Delta)$ and since $p_2\iota$ is an
isomorphism, we get $p_{2\#}\iota_\#s_0\big(\nu^\vee\big)=1_F$. Since
$p_{2\#}p_1^*= \pi^*\pi_\#$, we obtain
 \begin{equation*}
 \pi'_\# \big((v-ie)^a(w_1+e)^b\big(w_2-e^2\big)^c\big)
 =\pi^*\pi_\#v^a w_1^b w_2^c+Q(i;a,b,c).
 \end{equation*}
 Thus we have
 \begin{equation}\label{derive}
 y_i(a,b,c)=\beta^*_i(\pi^*y(a,b,c)+Q(i;a,b,c)).
 \end{equation}
 \end{sbs}

\section{The main theorem}\label{sc:thm}
 Fix a smooth projective family of surfaces
$\pi\colon F\to Y$, and a relative effective divisor $D$ on $F/Y$. For each
$r\ge1$, we introduce a natural cycle $U(D,r)$ on $Y$ that enumerates
the fibers $D_y$ with~$r$ nodes. Our first goal is to prove
Proposition~\ref{prp611}, which gives a recursive relation for the class
$u(D,r):=[U(D,r)]$ in terms of the classes $u(D_i,r_i)$ of the induced
pairs $(F_i/X_i,D_i)$ introduced in Section~\ref{sbsIndTr}. This
relation is the key to the proof of our main theorem, Theorem~\ref{th:5-1}.

\begin{dfn}\label{dfnUr} Fix $r\ge1$. Form the direct image on $Y$ of
the fundamental cycle $[G(F/Y,D;r)]$, remove the part supported in
$Y(\infty)$, and denote the result by $U(D,r)$. In other words,
$U(D,r)$ is obtained as follows. For each generic point $z$ of
$G(F/Y,D;r)$, let $n_z$ be the length of $\mc O_z$ over~$\mc O_{q(z)}$
provided this length is finite and $q(z)\notin Y(\infty)$; otherwise,
let $n_z$ be $0$. Let $\overline{\{q(z)\}}$ be the closure of
$\{q(z)\}$.
 \begin{equation*}\label{equUr}
 U(D,r):=\sum_zn_z\overline{\{q(z)\}}.
 \end{equation*}
In addition, let $U(D,0)$ denote the fundamental cycle of $Y$, and set
$U(D,-r):=0$.

Finally, set $u(D,r):=[U(D,r)]$. It's a class on $Y$.
\end{dfn}

\begin{prp}\label{prp610}
 Fix $r\ge1$. Assume that
the pair $(F/Y,D)$ is $r$-generic.
Then $U(D,r)$ has pure codimension $r$, and its support is just
the closure of $Y(r\I A_1)$.
\end{prp}

\begin{proof}
This follows from the second part of Lemma~\ref{lemCnstr}, with $\I D=r\I A_1$:
Let $z$ be a generic point of $G(F/Y,D;r)$. Then
$\cod_z(G(F/Y,D;r),H(r)) \le 3r$ because $\Hilb^{3r}_{D/Y}$ is the
zero scheme of a section of a locally free sheaf of rank $3r$ on
$\Hilb^{3r}_{F/Y}$ by \cite[Proposition~4, p.~5]{AIK}.
\end{proof}

\begin{prp}\label{prp611} Fix $r\ge1$. Then the following formula holds:
\begin{equation}\label{eq:4-41}
u(D,r) = \frac{1}{r} \pi_\#\sum_{i\ge 2}(-1)^i {\beta_i}_\#u(D_i,r_i),
\end{equation}
 where $\beta_i \colon X_i \to F$ denote the inclusions.
\end{prp}

\begin{proof}
 Notice that the sum in \eqref{eq:4-41} is finite, as
$r_i=r-\binom{i+1}{2}+2$ by \eqref{eq61b} and as $U(D,s)=0$ for $s<0$ by
Definition~\ref{dfnUr}.

 Let's first explain set-theoretically why the formula should hold.
Consider a closed point $y\in Y(r)$. The curve $D_y$ has precisely $r$
nodes, and if we blow up one of the nodes, $x\in X_2$ say, the strict
transform $(D_2)_x$ has $r-1$ nodes. Hence, above $y\in Y(r)$, we get
$r$ points $x\in X_2(r-1)$. But not all $r-1$-nodal curves of $D_2/X_2$
arise in this way: if $x\in X_3$ is an (ordinary) triple point of
$D_{\pi(x)}$, then $(D_2)_x=(D_3)_x+E_x$, hence it has the 3 nodes
$(D_3)_x\cap E_x$. Therefore, set-theoretically, $X_2(r-1)$ consists of
two parts, one mapping $r:1$ to $Y(r)$, the other mapping $1:1$ to $Y(\I
D_4+(r-4)\I A_1)$. The second part is equal to the part of $X_3(r-4)$
not contained in $X_4$. The part contained in $X_4$ is equal to
$X_4(r-8)$ minus the part contained in $X_5$, and so on.

The fact that this reasoning is valid on the cycle level is precisely
what Lemma~\ref{lemZiso2} shows. It~therefore only remains to show that
there is a natural map
\[
G(F_2/X_2,D_2;r-1)\smallsetminus G(F_3/X_3,D_3;r-4)\to G(F/Y,D;r)
\]
 which is $r:1$. If $\mc Z'\to T$ is a
$T$-point of $G(F_2/X_2,D_2;r-1)\smallsetminus G(F_3/X_3,D_3;r-4)$
(i.e., a family of $r-1$ double points in $D_2\subset F_2$ over $T$,
none of which lie on $E$), we send it to the image ${\vf'}(\mc Z'\cup
2E)$ in $F$ to get a family over $T$ of $r$ double points in $D$. This
map induces an $r:1$ map from the components of the cycle $U(D_2,r-1)$
that are supported on $X_2\smallsetminus X_3$ to $U(D,r)$.
 \end{proof}

Notice that the case $r=8$ of \eqref{eq:4-41} looks different from (4.6)
on p.~230 of \cite{KP99}.
Indeed, in~\cite{KP99}, we used $u(D_2,7)$ to denote
what we denote by $\frac{1}{7!} P_7(a_\bullet(D_2))\cap [X_2]$ here.
However, the mathematics is consistent.

Our main result is Theorem~\ref{th:5-1}, the first part of which was
proved in \cite{K--P}. We now prove the last part of the theorem, namely
the part concerning the expression for $u(D,r)=[U(D,r)]$, where $U(D,r)$
is the cycle introduced in Definition~\ref{dfnUr}.

Recall from Section~\ref{sc:IndFam} that for a given pair $(F/Y,D)$, the
subscheme $X_i\subset F$ denotes the scheme of zeros of the natural
section $\sigma_i$ of $\mathcal P^{i-1}_{F/Y}(D)$. After making some
simplifications, we prove the theorem when $X_4=\varnothing$. This proof
is easy, and it yields the case $r\le 7$. We then consider the case
$r=8$, which is more difficult due to the presence of nonreduced fibers
in codimension $r_2=7$ in the family of curves of the induced pair
$(F_2/X_2,D_2)$.

\begin{thm}[Main] \label{th:5-1}
Let $\pi\colon F\to Y$ be a smooth
 projective family of surfaces, and $D$ a relative effective divisor.
 Assume $Y$ is Cohen--Macaulay and equidimensional.
Fix an integer $r\ge 0$,
and assume
 \begin{enumerate}\itemsep=0pt
 \item[$(i)$] if\/ $Y(\infty)\ne\varnothing$, we have $\cod
 Y(\infty)\ge r+1$,
 \item[$(ii)$]
 the pair $(F/Y,D)$ is $r$-generic.
 \end{enumerate}
 Then either $Y(r\I A_1)$ is empty, or it has pure codimension $r$; in
either case, its closure $\overline{Y(r\I A_1)}$ is the support of a
natural nonnegative cycle $U(D,r)$.

Let $b_s(D)$ be the polynomial in $v$, $w_1$, $w_2$ output by \/ {\rm
Algorithm 2.3} in {\rm \cite{K--P}}, set $a_s(D):=\pi_\#b_s(D)$, and let $P_r$
be the $r$th Bell polynomial. Assume $r\le 8$, and if $r=8$, then
$(F/Y,D)$ is strongly $8$-generic. Then the rational equivalence class
$u(D,r):=[U(D,r)]$ is given by the formula
 \[
 u(D,r)= \frac{1}{r!}P_r(a_1(D),\ldots,a_r(D))\cap [Y].
 \]
 \end{thm}

\begin{proof}
 First of all, we may assume $Y(\infty)=\varnothing$. Indeed, $\cod
Y(\infty)\ge r+1$ by hypothesis. Hence we may replace $Y$ by
$Y\smallsetminus Y(\infty)$, and thus assume that all the fibres of $\pi
|_D\colon D\to Y$ are reduced.

Second, $\cod X_i=\binom{i+1}2$ for $i=2,3,4$ by Corollary~\ref{lm66}. Therefore, $X_i$ is a local complete intersection in $F$,
and $F$ is smooth over the Cohen--Macaulay scheme $Y$, hence is
Cohen--Macaulay, and so $X_i$ is too. Since $Y$ is equidimensional, so
is $F$, and hence so is $X_i$.

By \cite[Lemma~2.4, p.~73]{K--P} there are at most finitely many $\I
D$ such that $Y(\I D)$ is nonempty; hence, we may remove all $Y(\I D)$
with $\cod Y(\I D) \ge r+1$. If $x\in X_{i}$ is a closed point, $i\ge
5$, then $D_{\pi(x)}$ contains a point of multiplicity at least $i$;
hence, the minimal Enriques diagram $\I D$ of $D_{\pi(x)}$ satisfies
$\cod \I D \ge \binom{i+1}2-2\ge 13$ by the formula for $\cod \I D$ in
\cite[p.~217]{KP99}. Therefore, $\cod Y(\I D)\ge 13 > r$, since $r\le
8$. Hence $X_i=\varnothing$ for $i\ge 5$. If $x\in X_4$, then $x\in
D_{\pi(x)}$ has multiplicity $\ge 4$; hence, $\cod \I D \ge 8$. If
$\cod \I D \ge 9$, then $\cod Y(\I D)\ge 9$; so $Y(\I D) =
\varnothing$. Hence we have $\cod \I D = 8$. But the only diagram with a
root of multiplicity 4 and codimension~8 is the diagram $X_{1,0}$
corresponding to an ordinary quadruple point; see \cite[Figures~2--6, p.~218]{KP99}.

The recursive formula of Proposition~\ref{prp611} applies. It gives, for
$r\le 8$,
\[ ru(D,r)=\pi_\#\sum_{i=2}^{4}(-1)^i\beta_{i\#}u(D_i,r_i),
 \]
 where $r_2=r-1$, $r_3=r-4$ and $r_4=r-8$.

\begin{prp} The theorem holds if $X_4=\varnothing$ and
$Y(\infty)=\varnothing$.
 \end{prp}

\begin{proof}
The proof is by induction on $r$.
 For $r=1$, we have
\begin{align*}\label{alpr5.5}
 u(D,1) &= {\pi}_\#{\beta_2}_\#u(D_2,0) = \pi_{\#}[X_{2}]
 = \pi_{\#}x_{2}\cap [Y] \\
 &= \pi_\#b_{1}(D)\cap [Y]=a_1(D)\cap [Y]=P_1(a_1(D))\cap [Y].
\end{align*}

 Assume next that $r\ge 2$ and that the theorem holds for all
 families verifying the hypotheses of the theorem with $r$ replaced by
$r' < r$. In particular, the statement then holds for the induced pairs
$(F_i/X_i,D_i)$, for $i=2,3$, defined in
Section~\ref{sbsIndTr}. Indeed, $X_i(\infty)=\varnothing$, and
$(F_i/X_i,D_i)$ is $r_i$-generic by Proposition~\ref{lm:4-3}; that is, $(ii)$
of the theorem holds with $r$ replaced by $r_i$.

To simplify the notation, let us write
\begin{equation}\label{bell}
P_m(z_\bullet):=P_m(z_1,\dots,z_m),
\end{equation}
where $P_m$ is the
$m$th Bell polynomial and $z_1,\dotsc,z_m$ are variables. Then
we get
 \[
 r!u(D,r)=\pi_{\#}\bigl(\beta_{2\#}P_{r-1}(a_{\bullet}(D_{2}))
 -(r-1)!/(r-4)!\beta_{3\#}P_{r-4}(a_{\bullet}(D_{3}))\bigr)\cap [Y].
 \]
 By definition, $a_{s}(D_{i})=\pi_{i\#} b_{s}(D_{i})$. By~applying (\ref{derive}) to the polynomials $b_s(D)$ (cf.\ \cite[Algorithm~2.3]{K--P}),
\begin{equation*} \beta_{i\#}
P_{m}(a_{\bullet}(D_{i}))
=P_{m}(\pi^*a_{\bullet}(D)+Q(i,b_{\bullet}(D)))\cdot x_{i}(D).
\end{equation*}
By the binomial property of the Bell
polynomials \cite[equation~(4.9), p.~265]{Bell}, we have
 \begin{equation*}
 P_{m}(\pi^*a_{\bullet}(D)+Q(i,b_{\bullet}(D)))
 =\sum_{k=0}^m \binom{m}{k}
 P_{m-k}(\pi^*a_\bullet(D))P_{k}(Q(i,b_{\bullet}(D))).
 \end{equation*}
 Plugging this in and using the definition of $b_{s}(D)$ and $a
_{s}(D)$, we get
\[
r!u(D,r)=\sum_{k=0}^{r-1} \binom{r-1}{k}
P_{r-1-k}(a_\bullet(D))a_{k+1}(D)\cap [Y]=P_{r}(a_{\bullet}(D))\cap
[Y],
\] where the last equality follows from the recursive property of
the Bell polynomials \cite[equation~(4.2), p.~263]{Bell}.
 \end{proof}

When $r\le 7$, we can remove all $Y(\I D)$ with $\cod \I D\ge 8$. If $\I
D$ contains a root of multiplicity $\ge 4$, then $\cod \I D \ge
\binom{4+1}{2}-2=8$, hence we may assume $X_4=\varnothing$. This
proves the theorem for $r\le 7$.

Assume $r=8$. By~Proposition~\ref{prp611}, we have
\[
8 u(D,8)= \pi_{\#}\bigl( \beta_{2\#}u(D_{2},7)-
 \beta_{3\#}u( D_{3},4)+ \beta_{4\#}u( D_{4},0)\bigr).
 \]
 The induced pairs $(F_i/X_i,D_i)$, $i=3,4$ satisfy the conditions of
the theorem, with $r$ replaced by~$r_i$,
 hence, by the case $r\le7$ of the theorem:
 \[
 u(D_3,4)=\frac{1}{4!}P_4(a_\bullet(D_3)) \cap [X_3]
 \qquad \text{and}\qquad u(D_4,0)=[X_4].
 \]
 Note that, since $F$ is Cohen--Macaulay, $[X_i]=x_i \cap [F]$ with
$x_i$ as in Algorithm~2.3 in~\cite{K--P}. The induced pair
$(F_2/X_2,D_2)$ does not satisfy the conditions for $r$ replaced by
$r_2$. Indeed, note that $D_2|_{F_4}=(D-2E)|_{F_4} = (D-4E+2E)|_{F_4} =
D_4 + 2E|_{F_4}$ and that $D_4$ is relative effective on~$F_4/X_4$ by
Lemma~\ref{lm:4-2}; hence, $D_2$ has nonreduced fibers above~$X_4$. So
$X_2(\infty)=X_4$, and hence has codimension $r_2=7$ in~$X_2$.

However, from what we have seen above, if we restrict the family $F_2\to
X_2$ to $X_2\smallsetminus X_4$, then $\frac{1}{7!}P_7(a_\bullet(D_2))
\cap [X_2\smallsetminus X_4]$ is the class of the 7-nodal curves of that
family. So the difference
 \begin{equation}\label{eq:TheDiff}
 \frac{1}{7!}P_7(a_\bullet(D_2))\cap [X_2] - u(D_{2},7)
 \end{equation}
 is the \emph{correction term} we are looking for. It is the class of a
cycle of codimension 7, supported on the codimension 7 subscheme $X_4$
of $X_2$. As Theorem~\ref{constant} shows, \eqref{eq:TheDiff} is equal to
$C[X_4]$, where the constant $C$ is an integer, which is independent of
the given pair $(F/Y,D)$. Hence it suffices to compute $C$ in any
particular case; for example, in \cite[Example~3.8, p.~80]{K--P}, we
worked out the case of $8$-nodal quintic plane curves, and found $C =
3280$. Thus Theorem~\ref{th:5-1} is proved. \end{proof}

 \begin{rmk}
 Assume $(F/Y,D)$ is the direct sum of
 two pairs $(F'/Y, D')$ and $(F''/Y, D'')$ over the same base. Then
 the $r$-nodal curves of $D\to Y$ consists of the unions of the
$(r-i)$-nodal curves of $D'\to Y$ and
 the $i$-nodal curves of $D''\to Y$ for $i=0,\dots,r$. Hence, the
existence of a~universal polynomial for $r$-nodal curves implies that
the generating series for $(F/Y,D)$ is equal to the product of the
generating series for $(F'/Y,D')$ and $(F''/Y, D'')$. This fact was
observed by G\"ottsche in the case of a trivial family
\cite[p.~525]{gottsche}, and by Laarakker in the general case
 \cite[Section~5.1, p.~4935]{Laarakker}. In the case of a trivial family,
G\"ottsche observed that this multiplicativity implies that the
universal polynomials are Bell polynomials. However, as observed by
Laarakker, this conclusion does not follow in the case of a nontrivial
family.

Let $a_j(D)$, $a_j(D')$, $a_j(D'')$ be the classes, introduced in
Theorem~\ref{th:5-1}, for the three pairs. Clearly,
$a_j(D)=a_j(D')+a_j(D'')$. So (for $r\le 8$),
in the notation of~\eqref{bell},
 \[
 u(D,r)=\frac{1}{r!} P_r(a_\bullet(D)) \cap [Y]
 =\frac{1}{r!} P_r(a_\bullet(D')+a_\bullet(D'')) \cap [Y].
 \]
By the binomial property of the Bell polynomials,
 the right-hand side is equal to
 \[
 \sum_{i=0}^r \frac{1}{(r-i)!} P_{r-i}(a_\bullet(D')) \frac{1}{i!}P_{i}(a_{\bullet}(D''))\cap [Y].
 \]
 Hence the Bell polynomial shape of the universal polynomials is in
agreement with the multiplicative property of the generating series of $(F/Y,D)$.
 \end{rmk}

\section{An expression for the correction term}\label{sc:corr}
 We now find an expression for the correction term \eqref{eq:TheDiff}.
First, in Section~\ref{sb:schemes} we define some useful schemes.
Then in Lemma~\ref{lm:rec}, we give an expression for $u(D_2,7)$,
obtained via repeated use of the recursion formula of Proposition~\ref{prp611}. Then in Section~\ref{sb:ei}, we introduce classes
$e(W_i)$ on $X_2$ of cycles on $X_4$. Finally, in Proposition~\ref{correction}, we express \eqref{eq:TheDiff} as a linear combination
of the $e(W_i)$.

\begin{sbs}{}{\bf Some important schemes.} \label{sb:schemes}
 Let $\big(F_2^{(1)}/X_2^{(0)}, D_2^{(1)}\big):=
 (F_2/X_2, D_2)$ be the induced pair of $(F/Y,D)$. Define recursively
$\big(F_2^{(j+1)}/X_2^{(j)}, D_2^{(j+1)}\big)$ as the induced pair of
$\big(F_2^{(j)}/X_2^{(j-1)}, D_2^{(j)}\big)$. Let
$\big(F_3^{(j+1)}/X_3^{(j)}, D_3^{(j+1)}\big)$ be the induced pair of
$\big(F_2^{(j)}/X_2^{(j-1)}, D_2^{(j)}\big)$. Let $X_2\big(D_3^{(j+1)}\big) \subset
F_3^{(j+1)}$ be the zero scheme of the section of $\mc
P^1_{\pi_3^{(j+1)}}\big(D_3^{(j+1)}\big)$ induced by that defining the divisor
$D_3^{(j+1)}$.

Let $D^{(j+1)}-E^{(j)}$ denote the restriction of the divisor
$D^{(j+1)}-{\varphi^{(j+1)}}^{-1}E^{(j)}$ to
$F^{(j+1)}|_{X_2(D_3^{(j)})}$. This divisor is effective; so it
induces a section of the restriction of $\mc
P^1_{\pi^{(j+1)}}\big(D^{(j+1)} - {\varphi^{(j+1)}}^{-1}E^{(j)}\big)$. Let
$X_2\big(D^{(j+1)}-E^{(j)}\big)$ denote its scheme of zeros. Let
$D^{(j+2)}-E^{(j)}$ denote the restriction of $D^{(j+2)}-
{\varphi^{(j+2)}}^{-1}{\varphi^{(j+1)}}^{-1}E^{(j)}$ to
$F^{(j+2)}|_{X_2(D^{(j+1)}-E^{(j)})}$, and $X_2\big(D^{(j+2)}-E^{(j)}\big)$ the
scheme of zeros of the induced section of the restriction of $\mc
P^1_{\pi^{(j+2)}}\big(D^{(j+2)}-
{\varphi^{(j+2)}}^{-1}{\varphi^{(j+1)}}^{-1}E^{(j)}\big)$.

Form these five equidimensional schemes of dimension $\dim X_2-7$, or
$\dim X_4$:
 \begin{gather*}
 \mathfrak X_2^{(7)} :=
 \overline {X_2^{(7)} \smallsetminus X_2^{(7)}|_{X_4}},\\
 \mathfrak X_3^{(4)} :=
 \overline {X_3^{(4)}\smallsetminus X_3^{(4)}|_{X_4}},\\
 \mathfrak X_2\big(D_3^{(4)}\big) :=
 \overline {X_2\big(D_3^{(4)}\big)\smallsetminus X_2\big(D_3^{(4)}\big)|_{X_4}},\\
 \mathfrak X_2\big(D^{(4)}-E^{(3)}\big) :=
 \overline {X_2\big(D^{(4)}-E^{(3)}\big)\smallsetminus
 X_2\big(D^{(4)}-E^{(3)}\big)|_{X_4}},\\
 \mathfrak X_2\big(D^{(4)}-E^{(2)}\big) :=
 \overline {X_2\big(D^{(4)}-E^{(2)}\big) \smallsetminus X_2\big(D^{(4)}-E^{(2)}\big)|_{X_4}}.
\end{gather*}
For $j=2, \dots, 7$, consider the composed map $\pi^{(j)}
\circ\dotsb\circ \pi^{(1)}\colon F^{(j)}\to F$, and let $\pi_j\colon
F^{(j)}|_{X_2} \to X_2$ denote its restriction.
 \end{sbs}

 \begin{lem}\label{lm:rec}
 Then
 \begin{align*} u(D_2,7)={}&\frac{1}{7!}
\pi_{7\#}\big[ \mathfrak X_2^{(7)}\big]-\frac{3!}{7!}\pi_{4\#}\big[ \mathfrak
X_3^{(4)}\big]-\frac{4!}{7!}\pi_{4\#} \big[ \mathfrak X_2\big(D_3^{(4)}\big)\big] -\frac{5!}{7!2!}\pi_{4\#} \big[ \mathfrak X_2\big(D^{(4)}-E^{(3)}\big)\big]
\\
&-
\frac{6!}{7!3!}\pi_{4\#}\big[ \mathfrak X_2\big(D^{(4)}-E^{(2)}\big)\big].
 \end{align*}
 \end{lem}

\begin{proof}
 As $(F/Y,D)$ is 8-generic, the successive induced pairs are
$j$-generic (for appropriate~$j$) by Proposition~\ref{lm:4-3}. Thus the
lemma follows from repeated use of the recursion formula of
Proposition~\ref{prp611}.
 \end{proof}

\begin{sbs}{}{\bf The classes $\boldsymbol{e(W_i)}$.} \label{sb:ei}
 Next we find an expression for each term appearing in the formula for
$u(D_2,7)$ in Lemma~\ref{lm:rec}. We just consider $ \mathfrak
X_2^{(7)}$, since the other schemes can be studied in a~similar way and
their classes have similar expressions.

By definition, $X_2^{(j)}$ is the scheme of zeros of the section
$\sigma_2^{(j)}$ of $\mc P^1_{\pi_2^{(j)}}\big(D_2^{(j)}\big)$ induced by the
section $\sigma^{(j)}$ defining $D_2^{(j)}$. For $j=1,\dots,6$, set
 \begin{gather}
 \mathfrak X_2^{(j)}
 := \overline{X_2^{(j)}\smallsetminus X_2^{(j)}|_{X_4}},\qquad
 \mathfrak F_2^{(j+1)} := F_2^{(j+1)}|_{ \mathfrak X_2^{(j)}},\label{ga:F2}
 \\
 W_1^{(j)}:= \big(X_2^{(j)}|_{ \mathfrak X_2^{(j-1)}}\big)|_{X_4} \label{ga:W1j}.
 \end{gather}
 Then $X_2^{(j)}|_{ \mathfrak X_2^{(j-1)}} = \mathfrak X_2^{(j)}\cup
W_1^{(j)}$. Notice the $ \mathfrak F_2^{(j)}$ are equidimensional of
dimension $\dim X_2-j+3$ and that $\cod \big( \mathfrak X_2^{(j)},
 \mathfrak F_2^{(j)}\big) = 3$. Thus $\dim \mathfrak X_2^{(j)}=\dim X_2-j$.

To determine the dimensions of the ``excess schemes'' $W_1^{(j)}$,
consider the fibers of $W_1^{(j)}\to W_1^{(j-1)}\cap \mathfrak
X_2^{(j-1)}$. Starting with $x\in X_4$, we have
$\big(D_2^{(1)}\big)_x=\Gamma_x\cup 2E^{(1)}_x$, where $\Gamma_x$ is the strict
transform of $D_{\pi(x)}$ under the blowup of $F_{\pi(x)}$ at
$x$. Hence, a local calculation yields $\big(W_1^{(1)}\big)_x$, which is
$\big(X_2^{(1)}\big)_x$, is equal to $E^{(1)}_x$ with four embedded points at
the intersections of~$\Gamma_x$ with~$E^{(1)}_x$. Thus $\dim
W_1^{(1)}=\dim X_4+1$.

 Next take $z\in \big(W_1^{(1)}\big)_x$, but $z\notin \Gamma_x$. Then the fiber
$\big(W_1^{(2)}\big)_z$ is the strict transform of $E^{(1)}_x$ plus four
embedded points. If $z\in \Gamma_x\cap E^{(1)}_x$, then $\big(W_1^{(2)}\big)_z$
has an additional point, namely, the intersection of the strict
transform of $\Gamma_x$ with $E^{(2)}_z$. Hence $\dim W_1^{(2)}=\dim
X_4+2$. Continuing, we get
\[ \dim W_1^{(3)}=\dim X_4+3=\dim X_2-4=\dim X_2^{(3)}-1.\]
 Thus $W_1^{(j)}\subset X_2^{(j)}$ and $ \mathfrak
X_2^{(j)}=X_2^{(j)}$ for $j\le 3$.

 For $j=4$ we get $\dim W_1^{(4)}=\dim X_4+4=\dim X_2-3=\dim
X_2^{(4)}+1$. Then $\dim W_1^{(4)}\cap \mathfrak X_2^{(4)}\le \dim
X_2^{(4)}-1$. Hence $\dim W_1^{(5)}\le \dim X_2^{(4)}-1+1=\dim
X_2^{(5)}+1$. Continuing, we get $\dim W_1^{(j)}\le \dim X_2^{(j)}+1$
for $j\ge 4$.

To simplify notation, set $\mc P^{(j)}:=\mc
P^1_{\pi^{(j)}}\big(D^{(j)}\big)$. Consider $\mc P^{(7)}$ restricted to
$ \mathfrak F_2^{(7)}$. The scheme of zeros of its section
$\sigma_2^{(7)}$ is equal to $ \mathfrak X_2^{(7)}\cup W_1^{(7)}$. Blow
up $ \mathfrak F_2^{(7)}$ along $W_1^{(7)}$ and apply the residual
formula for top Chern classes \cite[Example~14.1.4, p.~245]{Fu}. After
pushing down to $ \mathfrak F_2^{(7)}$, we find
 \begin{equation}\label{eq:X7}
 \big[ \mathfrak X_2^{(7)}\big] = c_3\big(\mc P^{(7)}\big)\cap
 \big[ \mathfrak F_2^{(7)}\big] + \big\{c\big(\mc P^{(7)}\big)\cap
 s\big(W_1^{(7)}, \mathfrak F_2^{(7)}\big)\big\}_{\dim X_4}.
 \end{equation}
 Here is why \eqref{eq:X7} holds.

 Let $\sigma'$ denote the induced section of $\mc P^{(7)}$ twisted by
the ideal sheaf of the exceptional divisor on the blowup of $ \mathfrak
F_2^{(7)}$. Let $\mathbb Z(\sigma')$ denote the localized top Chern
class of the pullback of $\mc P^{(7)}$ with respect to~$\sigma'$. The
zero scheme $Z(\sigma')$ is equal to the strict transform of $ \mathfrak
X_2^{(7)}$, hence has codimension~3 in the blowup of $ \mathfrak
F_2^{(7)}$. It follows from \cite[Proposition~14.1(b), p.~244]{Fu} that
$\mathbb Z(\sigma')$ is the class of a positive cycle with support
$Z(\sigma')$. Since the blowup of $ \mathfrak F_2^{(7)}$ need not be
Cohen--Macaulay, we cannot immediately conclude that $\mathbb
Z(\sigma')=[Z(\sigma')]$. However, since $ \mathfrak
F_2^{(7)}\smallsetminus \mathfrak F_2^{(7)}|_{X_4}$ is Cohen--Macaulay,
the restrictions of $\mathbb Z(\sigma')$ and $[Z(\sigma')]$ agree above
$ \mathfrak X_2^{(7)}\smallsetminus \mathfrak X_2^{(7)}|_{X_4}$; hence
they are equal.

 Since $\big[ \mathfrak F^{(7)}_2\big]=\big[F^{(7)}_2|_{ \mathfrak
X^{(6)}_2}\big]=\pi^{(7)*}\big[ \mathfrak X_2^{(6)}\big]$, the pushdown of the first
term under $\pi^{(7)}$ gives $\pi^{(7)}_\# c_3\big(\mc P^{(7)}\big)\cap
\big[ \mathfrak X_2^{(6)}\big]$ by the projection formula. We then replace $
\big[\mathfrak X_2^{(6)}\big]$ by the analogue of equation~(\ref{eq:X7}) and
push down the resulting terms by~$\pi^{(6)}$.

 Continuing this way, we get a formula for $\pi_{7\#}\big[ \mathfrak
X_2^{(7)}\big]$. To simplify notation, set
 \begin{gather*}
 d_j(W_1) := \big\{c\big(\mc P^{(7-j)}\big)\cap
 s\big(W_1^{(7-j)}, \mathfrak F_2^{(7-j)}\big)\big\}_{\dim X_4 +j}
 \qquad\text{for}\quad j=0,\dotsc,3,
 \\
 e(W_1) := \pi^{(1)}_\# \cdots \pi^{(7)}_\#d_0(W_1)
 + \pi^{(1)}_\# \cdots \pi^{(6)}_\#
 \big(\pi^{(7)}_\# c_3\big(\mc P^{(7)}\big)d_1(W_1)\big)\notag
 \\ \hphantom{e(W_1) :=}
 {} + \pi^{(1)}_\# \cdots \pi^{(5)}_\# \big(\pi^{(6)}_\# \big(\pi^{(7)}_\#
 c_3\big(\mc P^{(7)}\big) c_3\big(\mc P^{(6)}\big)d_2(W_1)\big)\big)\notag
 \\ \hphantom{e(W_1) :=}
 {} + \pi^{(1)}_\# \cdots \pi^{(4)}_\#\big(\pi^{(5)}_\#
 \big(\pi^{(6)}_\#\big(\pi^{(7)}_\# c_3\big(\mc P^{(7)}\big) c_3\big(\mc
 P^{(6)}\big)c_3\big(\mc P^{(5)}\big)\big)d_3(W_1)\big)\big).\notag
 \end{gather*}
 Note that the $d_j(W_1)$ are classes in $A^*\big(W_1^{(7-j)}\big)$, and that
restricting the map $\pi_{7-j}$ gives a proper map $W_1^{(7-j)} \to X_4$.

 Then the resulting formula for $\pi_{7\#}[ \mathfrak
X_2^{(7)}]$ is the following:
 \begin{equation*}\label{alpi7}
 \pi_{7\#}\big[ \mathfrak X_2^{(7)}\big] = \pi^{(1)}_\# \big(\pi^{(2)}_\#\big(\dotsb
 \big(\pi^{(7)}_\# c_3\big(\mc P^{(7)}\big)\dotsb\big) c_3\big(\mc
 P^{(2)}\big)\big) c_3\big(\mc P^{(1)}\big)\big) \cap [X_2] + e(W_1).
 \end{equation*}

Similarly, we obtain formulas for the classes $\pi_{4\#}\big[ \mathfrak
X_3^{(4)}\big]$ and $\pi_{4\#}\big[ \mathfrak X_2\big(D_3^{(4)}\big)\big]$ and $\pi_{4\#}\big[
\mathfrak X_2\big(D^{(4)}-E^{(3)}\big)\big]$ and $\pi_{4\#}\big[ \mathfrak
X_2\big(D^{(4)}-E^{(2)}\big)\big]$. For $i=2,\dots,5$, define the classes $e(W_i)$
on $X_2$ correspondingly.
 \end{sbs}

\begin{prp}\label{correction} The correction term~\eqref{eq:TheDiff} is equal to
 \begin{gather*}
\frac{1}{7!}P_7(a_\bullet(D_2))\cap [X_2]\!-u(D_2,7) = \frac{1}{7!}
e(W_1)\!-\frac{3!}{7!}e(W_2)\!-\frac{4!}{7!}e(W_3) \!-\frac{5!}{7!2!}e(W_4)\!-
\frac{6!}{7!3!}e(W_5).
 \end{gather*}
 \end{prp}

\begin{proof}
 Recall that the classes $a_s(D_2)$ on $X_2$ are obtained
by pushing down the classes $b_s(D_2)$ on $F_2$ obtained by applying
Algorithm~2.3 of \cite[p.~72]{K--P} to the pair $(F_2/X_2, D_2)$. In
the case that $X_4=\varnothing$, the Algorithm would have produced the
formula $u(D_2,7)=\frac{1}{7!}P_7(a_\bullet(D_2))\cap [X_2]$.
Removing the classes $e(W_i)$, we get
 \begin{align*}
\frac{1}{7!}P_7(a_\bullet(D_2))\cap [X_2]={}&
\frac{1}{7!} \bigl(\pi_{7\#}\big[ \mathfrak
X_2^{(7)}\big]-e(W_1)\bigr)-\frac{3!}{7!}\bigl(\pi_{4\#}\big[ \mathfrak
X_3^{(4)}\big]- e(W_2)\bigr)
\\
&-\frac{4!}{7!}\bigl(\pi_{4\#}\big[ \mathfrak X_2\big(D_3^{(4)}\big)\big] - e(W_3)\bigr)\\
&
 - \frac{5!}{7!2!}\bigl(\pi_{4\#}\big[ \mathfrak
X_2\big(D^{(4)} - E^{(3)}\big)\big] - e(W_4)\bigr)
\\
& -\frac{6!}{7!3!}\bigl(\pi_{4\#}\big[ \mathfrak
X_2\big(D^{(4)}-E^{(2)}\big)\big]-e(W_5)\bigr).
 \end{align*}
 Lemma~\ref{lm:rec} now yields the asserted formula.
 \end{proof}

 \section{Independence of the correction term}\label{sc:indep}
 In this section, we prove Theorem~\ref{constant}, which asserts that the
correction term \eqref{eq:TheDiff} is equal to~$C[X_4]$, where $C$ is
independent of the strongly 8-generic pair $(F/Y,D)$ with
$Y(\infty)=\varnothing$. We~work locally analytically on $F$ at a general
closed point $x$ in $X_4$. Section~\ref{sbSetup} describes the local
setup. Lemma~\ref{lm:proper} asserts that locally we have the
properness we need to pushdown classes. Lemma~\ref{local} asserts that
the key classes $e(W_i)$ pull back to their local counterparts $e\big(\wh
W_i\big)$.

Lemma~\ref{lm:versal} asserts that the coefficient in $e\big(\wh W_i\big)$ of
$\big[\wh X_4\big]$ depends only on the analytic type of the ordinary quadruple
point $x\in D_{\pi(x)}$; namely, on the cross ratio of the four tangents
at~$x$. Its~proof requires $(F/Y,D)$ to be strongly 8-generic.
Finally, we prove Theorem~\ref{constant} by exhibiting a pair $(F/Y,D)$,
where~$X_4$ is irreducible and where any given value of the cross ratio
appears at some $x\in X_4$.

\begin{sbs}{}{\bf The local setup.} \label{sbSetup}
 Fix an 8-generic pair $(F/Y,D)$ with $Y(\infty)=\varnothing$, and a
general closed point $x\in X_4$. By~{\it general\/}, we mean that $x$
lies on a single irreducible component $Z$ of $X_4$ and that $x$ is an
ordinary quadruple point of $D_{\pi(x)}$. Let us arrange for every $x
\in X_4$ to be general as follows. First, if two components $Z'$ and
$Z''$ of $X_4$ meet, then $\dim (Z' \cap Z'') < \dim X_4$. So~$\cod
\pi(Z'\cap Z'')> \cod \pi (X_4)$. But $\cod(\pi(X_4),Y) = 8$ by
\eqref{eq43d} as $(F/Y,D)$ is 8-generic. Hence we may discard $Z' \cap
Z''$. Second we may discard the locus of $y \in Y$, where $D_y$ has a~singularity~$x$ worse than an ordinary quadruple point, again because
$(F/Y,D)$ is 8-generic.

Set $\wt F:= \Spec \mc O_{F,x}$ and $\wt Y:= \Spec \mc O_{Y,\pi(x)}$ and
$\wt D:= \Spec \mc O_{D,x}$. Denote the induced pair of $\big(\wt F/\wt
Y, \wt D\big)$ by $\big(\wt F_2/\wt X_2, \wt D_2\big)$. The bundles of relative
principal parts are compatible not only with the base change $\wt Y\to
Y$, but also with the maps $\wt F\to F$ and $\wt F^{(j)} \to F^{(j)}$;
cf.\ \cite[Proposition~16.4.14, p.~22]{EGAIV4}. So although the $\wt X_i$
for $i\ge 2$ are defined in terms of $\big(\wt F/\wt Y, \wt D\big)$, we have
$\wt X_i = \Spec \wt{\mc O}_{X_i,x}$. Similarly, the schemes
constructed in Section~\ref{sc:corr} for $(F_2/X_2, D_2)$ induce the
corresponding schemes for $\big(\wt F_2/\wt X_2, \wt D_2\big)$. Denote by
$\wt W^{(j)}_i$ the scheme corresponding to $W^{(j)}_i$; see~\eqref{ga:W1j}.

Next consider the completions of the local rings, giving us a pair $\big(\wh
F/\wh Y, \wh D\big)$. Replacing the principal parts bundles by their
completions, cf.~\cite[Example~16.14, p.~416]{Eisenbud}, construct the
$\wh X_i$, the induced pair $\big(\wh F_2/\wh X_2,\wh D_2\big)$, and the
corresponding schemes of Section~\ref{sc:corr}. Since the complete
principal parts bundles are pullbacks, $\wh X_i = \Spec \wh{\mc
O}_{X_i,x}$, and all the schemes of Section~\ref{sc:corr} for
$(F_2/X_2,D_2)$ pull back to the corresponding schemes for $\big(\wh
F_2/\wh X_2, \wh D_2\big)$. Denote by $\wh W^{(j)}_i$ the scheme
corresponding to $\wt W^{(j)}_i$, so to $W^{(j)}_i$.

The classes $e(W_i)$ of Section~\ref{sb:ei} are sums of pushdowns of
classes on the $W^{(j)}_i$. By~Lemma~\ref{lm:proper} below, the $\wh
W^{(j)}_i$ are proper over $\wh X_2$; hence, we may form the
corresponding classes $e(\wh W_i)$ for the pair $\big(\wh F_2/\wh X_2, \wh
D_2\big)$. Denote them by $e(\wh W_i)$.

Let $\epsilon\colon \wh X_2 \to X_2$ denote the composition of the flat maps
$\wh X_2\to \wt X_2$ and $\wt X_2 \to X_2$. Then $\big[\wh X_4\big] =
\epsilon^*[X_4]$, and Lemma~\ref{local} asserts $e(\wh
W_i)=\epsilon^*e(W_i)$.

Each $e(W_i)$ is the class of a cycle $U_i$ on $X_4$ of dimension $\dim
X_4$. Say that the component~$Z$ of~$X_4$ containing $x$ appears in
$U_i$ with coefficient $C'_i$ and in the fundamental cycle $|X_4|$ with
coefficient $C''_i$. Set $C_i := C_i'/C_i''$. Then the cycles $U_i$
and $C_i|X_4|$ become equal after restriction to a neighborhood of $Z$,
so the classes $e(W_i)$ and $C_i[X_4]$ do too. Thus
 \begin{equation}\label{eqwhWiX4}
 e\big(\wh W_i\big) = C_i\big[\wh X_4\big];
 \end{equation}
 furthermore, $C_i$ is independent of the choice of $x$ in~$Z$.
 \end{sbs}

\begin{lem}\label{lm:proper}
 The schemes $\wh W^{(j)}_i$ are proper over $\wh X_2$.
 \end{lem}

\begin{proof}
 Let us first show that the schemes $\wt W^{(j)}_i$ and $W^{(j)}_i|_{
\wt X_2}$ have the same support. It suffices to consider only the
$W^{(j)}_1$ as the other cases are similar.

Let $E_{j}\subset F^{(j)}$ be the union of the strict transforms
of the exceptional divisors $E^{(1)},\dots,E^{(j)}$; see
\cite[Definition~3.5, p.~423]{KP09}. It follows from the description in
Section~\ref{sc:corr} that $W^{(j)}_1$ is supported in $E_j$. The fibers
of the exceptional divisor of $\wt F_2$ are the same as the
corresponding fibers of the exceptional divisor of $F_2$. Hence, above
$x \in \wt X_2$, the fiber of $E^{(j)}$ lies in $\wt F_2^{(j)}$. Thus
$W^{(j)}_1|_{ \wt X_2}$ has support in $\wt F^{(j)}_2$, and it is the
same as the support of $\wt W^{(j)}_i$; moreover, this support is proper
over $\wt X_2$.

For $w\in W^{(j)}_i|_{ \wt X_2}$, the map $\mc O_{F^{(j)}_2,w}\onto \mc
O_{W^{(j)},w}$ pulls back to $\mc O_{\wt F^{(j)}_2,w}\onto \mc O_{\wt
W^{(j)},w}$. But clearly $\mc O_{F^{(j)}_2,w}=\mc O_{\wt F^{(j)}_2,w}$;
hence $\mc O_{W^{(j)}_i,w}= \mc O_{\wt W^{(j)_1},w}$. Hence $\wt
W^{(j)}_i=W^{(j)}_i|_{ \wt X_2}$. Therefore, $\wt W^{(j)}_i$ is proper
over $\wt X_2$. Thus the pullback $\wh W^{(j)}_i$ is proper over $\wh
X_2$.
 \end{proof}

\begin{lem}\label{local}
 We have $e\big(\wh W_i\big)=\epsilon^*e(W_i)$.
 \end{lem}

\begin{proof}
 It suffices to check that each summand of $e(W_i)$
pulls back to the corresponding summand of $e\big(\wh W_i\big)$. Here we
only consider the first summand of $e(W_1)$, since the other cases
are similar.

In \eqref{ga:F2}, we defined the schemes $\overline F_2^{(j)}$. Denote
by $\epsilon_2^{(j)}\colon \overline{\wh F}_2^{(j)}\to \overline F_2^{(j)}$
the map induced by the map $\epsilon$ defined in Section~\ref{sbSetup}. Notice that, as $\epsilon_2^{(7)}$ is flat,
\[
s\big(\wh W_1^{(7)},\overline {\wh
 F}_2^{(7)}\big)=s\big(\epsilon_2^{(7)-1}W_1^{(7)},\epsilon_2^{(7)-1}\overline
 F_2^{(7)}\big)=\epsilon_2^{(7)*} s\big(W_1^{(7)},\overline {F}_2^{(7)}\big);
 \]
 cf.\ \cite[Proposition~4.2(b), p.~74]{Fu}. But $\wh {\pi}_\#^{(j)}
\epsilon_2^{(j)*} = \epsilon_2^{(j-1)*} \pi_\#^{(j)}$ by~\eqref{A23}.
Thus
 \begin{align*}
 e\big(\wh W_i\big) &=\wh\pi^{(1)}_\# \dotsb \wh \pi^{(7)}_\#
 \big\{c\big(\wh {\mc P}^{(7)}\big)\cap s\big(\wh W_1^{(7)},\overline {\wh
 F}_2^{(7)}\big)\big\}_{\dim X_4}\\
 &=\wh\pi^{(1)}_\# \cdots \wh \pi^{(7)}_\# \big\{\epsilon_2^{(7)*} c\big(
 \mc P^{(7)}\big)\cap \epsilon_2^{(7)*} s\big(W_1^{(7)},\overline
 {F}_2^{(7)}\big)\big\}_{\dim X_4}\\
 &=\wh\pi^{(1)}_\# \cdots \wh \pi^{(6)}_\#
 \epsilon_2^{(6)*}\pi_\#^{(7)} \big\{c\big( \mc P^{(7)}\big)\cap
 s\big(W_1^{(7)},\overline {F}_2^{(7)}\big)\big\}_{\dim X_4}\\
 &=\dotsb =\epsilon^*\pi^{(1)}_\# \cdots \pi^{(7)}_\# \big\{c\big( \mc
 P^{(7)}\big)\cap s\big( W_1^{(7)},\overline F_2^{(7)}\big)\big\}_{\dim X_4} \\
 &= \epsilon^*e(W_i),
\end{align*}
 as desired.
 \end{proof}

 \begin{lem}\label{lm:versal}
 Assume $(F/Y,D)$ is strongly $8$-generic. Then $C_i$ depends just on
the analytic type of $D_{\pi(x)}$ at $x$, but is otherwise inpenendent
of the choice of $(F/Y,D)$.
 \end{lem}

\begin{proof}
 Recall that $x$ is an ordinary quadruple point of $\wh D_{\wh \pi(x)}$.
Let $(\bb V/\bb B, \bb D)$ be its versal deformation; see
\cite[Example~14.0.1, p.~101 and Theorem~14.1, p.~103]{hartshorne}.
Recall how $(\bb V/\bb B, \bb D)$ is constructed. Take variables
$t_1,\dots,t_9, u, v$. Identify $\wh F_{\pi(x)}$ with $\Spec k[[u,v]]$.
Say $\wh D_{\pi(x)}$ is defined by $f(u,v)$ in $k[[u,v]]$, and choose
$g_1,\dots,g_9$ in $k[u,v]$ whose classes in $k[[u,v]]/(f,f_u,f_v)$ form
a basis of that vector space. Then
\[
\bb B:= \Spec k[t_1,\dots,t_9] \qquad \text{and}\qquad
 \bb D:= \Spec B[[u,v]]\big/\Bigl(f+\sum t_ig_i\Bigr).
 \]
 Note that
$(\bb V/\bb B, \bb D)$ depends just on the analytic type of
$D_{\pi(x)}$ at $x$.

 Since $x\in \wh D_{\pi(x)}$ is an ordinary quadruple point,
$f=f_4+f_5+\cdots$, where $f_4$ is a product of independent linear forms.
Choose the $g_i$ so that only $g_9\in (u,v)^4$. Define $\bb B_4$
by the vanishing of $t_1,\dotsc,t_8$. Then $b\in \bb B$ lies in $\bb
B_4$ iff the fiber $\bb D_b$ has a quadruple point.

Recall from \cite[Theorem~14.1, p.~103]{hartshorne} that there exists a
map $\delta \colon \wh Y\to \bb B$ such that $\wh D$ and $\bb D \times_{\bb
B} \wh Y$ become isomorphic after completion along their fibers over
$\wh\pi(x)$. Since $\wh D$ is complete at~$x$, it is already complete
along its fiber. Form the completions $\wh{\bb V}$, $\wh{\bb B}$, $\wh{\bb
D}$, $\wh{\bb B}_4$ at the origin~$b_0$ of~$\bb B$. Then $\delta \colon \wh
Y\to \bb B$ factors through a map $\wh\delta \colon \wh Y\to \wh{\bb B}$,
and $\wh D$ is isomorphic to the completion of $\wh{\bb D}
\times_{\wh{\bb B}} \wh Y$ along its fiber over $\wh\pi(x)$. Each
subscheme $\wh X_i$ of $\wh D$ is the pullback of the corresponding
subscheme of $\wh{\bb D} \times_{\wh{\bb B}} \wh Y$, which, in turn, is
the pullback of the corresponding subscheme $\wh{\bb X}_i$ of $\wh{\bb
D}$.

Let us show $\wh\delta \colon \wh Y\to \wh{\bb B}$ is flat. Notice
$\wh\delta\big(\wh\pi\big(\wh X_4\big)\big) \subset \wh{\bb B}_4$. But
$\wh\delta\big(\wh\pi\big(\wh X_4\big)\big) \neq \{b_0\}$ because $(F/Y,D)$ is strongly
8-generic. However, $\dim \wh{\bb B}_4 = 1$. It follows that
$\cod\big(\delta^{-1}(b_0), \wh\pi\big(\wh X_4\big)\big) =1$. Now, $\cod\big(\wh\pi(\wh
X_4), \wh Y\big) = 8$ by \eqref{eq43d} since $(F/Y,D)$ is 8-generic. So
$\cod\big(\wh\delta^{-1}(b_0), \wh Y\big) =9$. But $\wh Y$ is
Cohen--Macaulay, and $\wh{\bb B}$ is smooth. Thus $\wh \delta$ is flat
by \cite[Proposition~15.4.2, p.~230]{EGAIV3}.

Form the class $e\big(\wh{\bb W}_i\big)$ for $\big(\wh{\bb V}/\wh{\bb B}, \wh{\bb
D}\big)$ analogous to the class $e\big(\wh W_i\big) $ for $\big(\wh F/\wh Y, \wh
D\big)$. The map $\wh{\bb X}_2 \times_{\wh{\bb B}} \wh Y \to \wh{\bb X}_2$
is flat, as it is induced by $\wh\delta$. Since the map $\wh X_2 \to \bb
X_2 \times_{\bb B} \wh Y$ is also flat, we can argue as in the proof of
Lemma~\ref{local} to conclude that $e\big(\wh{\bb W}_i\big)$ pulls back to~$e\big(\wh W_i\big)$. Owing to the same flatness, the fundamental class
$\big[\wh{\bb X}_4\big]$ on $\wh{\bb X}_2$ pulls back to the fundamental class~$\big[\wh X_4\big]$ on~$\wh X_2$.

Form the equation $e\big(\wh{\bb W}_i\big) = \bb C_i\big[\wh{\bb X}_4\big]$ on $\wh{\bb
X}_2$ analogous to $e\big(\wh W_i\big) = C_i\big[\wh X_4\big]$ on $\wh X_2$, see~\eqref{eqwhWiX4}. The~former equation pulls back to the latter owing to
the preceding paragraph. Thus $C_i = \bb C_i$. But~$\bb C_i$ depends
just on $(\bb V/\bb B,\bb D)$, so just on the analytic type of
$D_{\pi(x)}$ at $x$. Thus $C_i$ depends just on the analytic type of
$D_{\pi(x)}$ at $x$.
 \end{proof}

\begin{thm}\label{constant}
 Assume $(F/Y,D)$ is strongly $8$-generic. Then the correction term
{\rm\eqref{eq:TheDiff}} is equal to $C[X_4]$, where $C$ is independent of
the choice of $(F/Y,D)$.
 \end{thm}

 \begin{proof}
 By Lemma~\ref{lm:versal}, each $C_i$ depends just on the analytic
type of $D_{\pi(x)}$ at $x$; that is, on the cross ratio of the four
tangents at $x$. By~the last line in Section~\ref{sbSetup},
furthermore, $C_i$ is independent of the choice of $x$ in $Z$. Below,
we exhibit a pair where $X_4$ is irreducible and where any given value
of the cross ratio appears at some $x\in X_4$. It follows that $C_i$ is
independent of the choice of $(F/Y,D)$. Finally, Proposition~\ref{correction} now implies that $C$ is independent too, as desired

To build the pair, say $k$ is the base field, take variables
$t_1,\dotsc,t_8,t,u,v$, and set
 \begin{gather*}
 \mc A:=k\biggl[t_1,\dotsc,t_8,t, \frac{1}t,\frac{1}{t-1}\biggr], \qquad
 \mc B:=A[u,v],\qquad \mc C:=\mc B/(g),
 \end{gather*}
 where
\begin{gather*}
 g:=t_1+t_2u+t_3v+t_4u^2+t_5uv+t_6v^2+t_7u^2v+t_8uv^2+uv(u-v)(u-tv).
 \end{gather*}
 Set $Y:=\Spec \mc A$ and $F := \mathbb P^2_k\times Y$ and $D :=
\overline{\Spec\mc C}$.

 Then $X_4\subset \Spec \mc B$. Its ideal $\mc I$ is generated by the
partial derivatives with respect to $u$ and $v$ of $g$ up to order
three; so $I = (t_1,\dots,t_8,u,v)$. It follows that $X_4= \Spec (\mc
B/\mc I)=\Spec k\big[t, \frac{1}t,\frac{1}{t-1}\big]$. Thus $X_4$ is
irreducible.

Given $c\in k$, let $x\in X_4$ be the point with $t=c$. Then the fiber
$D_{\pi(x)}\subset \mathbb P^2_k$ is equal to the four lines
$uv(u-v)(u-cv)$ through $(0:0:1)$ with cross ratio equal to $c$. Thus
all cross ratios appear in this family, as desired.
 \end{proof}

\subsection*{Acknowledgements}
 Thanks are due to the referee for pointing out our inadvertent
change of notation from~\cite{KP99}, which is discussed immediately
after the proof of Proposition~\ref{prp610}.


\pdfbookmark[1]{References}{ref}
\LastPageEnding

\end{document}